\documentclass[a4paper,12pt]{amsart}

\usepackage[margin=3cm]{geometry} 
\usepackage{amsmath,amssymb,amsxtra,amscd,mathrsfs}
\usepackage{graphicx}
\usepackage[colorlinks=true, allcolors=blue]{hyperref}
\usepackage[all]{xy}
\usepackage{bm}
\usepackage{comment}

\numberwithin{equation}{section}

\newtheorem{theorem}{Theorem}[section]
\newtheorem{lemma}[theorem]{Lemma} 
\newtheorem{proposition}[theorem]{Proposition} 
\newtheorem{corollary}[theorem]{Corollary}

\theoremstyle{definition}
 
\newtheorem{notation}[theorem]{Notation} 
\newtheorem{remark}[theorem]{Remark} 
\newtheorem{example}[theorem]{Example}

\newcommand{\C}{\mathbb{C}} 
\newcommand{\Z}{\mathbb{Z}}
\newcommand{\N}{\mathbb{N}}
\newcommand{\G}{\mathbb{G}}
\newcommand{\I}{\mathbb{I}}
\newcommand{\J}{\mathbb{J}}

\newcommand{\PP}{\mathbb{P}} 
\newcommand{\Q}{\mathbb{Q}} 
\newcommand{\R}{\mathbb{R}}
\newcommand{\bbS}{\mathbb{S}}
\newcommand{\T}{\mathbb{T}}
\newcommand{\W}{\mathbb{W}}

\newcommand{\bc}{\boldsymbol{c}} 
\newcommand{\bone}{\mathbf{1}} 

\newcommand{\bt}{\mathbf{t}}
\newcommand{\btau}{{\bm{\tau}}}

\newcommand{\cH}{\mathcal{H}}
\newcommand{\cL}{\mathcal{L}}
\newcommand{\cM}{\mathcal{M}}
\newcommand{\cO}{\mathcal{O}}

\newcommand{\ovcM}{\overline{\cM}}

\newcommand{\adj}{\operatorname{adj}}
\newcommand{\ev}{\operatorname{ev}} 
\newcommand{\gt}{{\mathfrak{g}/\mathfrak{t}}}
\newcommand{\loc}{{\operatorname{loc}}} 
\newcommand{\newsigma}{\varsigma}
\newcommand{\pol}{{\operatorname{pol}}}
\newcommand{\pt}{\operatorname{pt}} 
\newcommand{\rank}{\operatorname{rank}}

\newcommand{\vir}{{\operatorname{vir}}} 
\newcommand{\NE}{\operatorname{NE}_{\N}} 
 
\newcommand{\End}{\operatorname{End}} 
\newcommand{\Fab}{F_{\operatorname{ab}}} 
\newcommand{\Fl}{\operatorname{Fl}} 
\newcommand{\FlG}{\operatorname{Fl}} 
\newcommand{\FlT}{\operatorname{Fl}_{\operatorname{ab}}} 
\newcommand{\Frac}{\operatorname{Frac}} 
 
\newcommand{\Gr}{\operatorname{Gr}} 
\newcommand{\GL}{\operatorname{GL}}
\newcommand{\Hom}{\operatorname{Hom}}
\newcommand{\Hyp}{\operatorname{Hyp}}

\newcommand{\Prin}{\operatorname{Prin}}
\newcommand{\QH}{\operatorname{QH}}
\newcommand{\Res}{\operatorname{Res}}

\def\corr#1{\left\langle#1 \right\rangle} 
\def\parfrac#1#2{\frac{\partial #1}{\partial #2}}

\title{A mirror theorem for partial flag bundles}

\author{Ionu\c t Ciocan-Fontanine} 
\email{ciocan@gate.sinica.edu.tw}
\address{Institute of Mathematics, Academia Sinica, Astronomy-Mathematics Building, No.\ 1, Sec.\ 4, Roosevelt Road, Taipei 10617, Taiwan.}

\author{Yuki Koto} 
\email{ykoto@gate.sinica.edu.tw}
\address{Institute of Mathematics, Academia Sinica, Astronomy-Mathematics Building, No.\ 1, Sec.\ 4, Roosevelt Road, Taipei 10617, Taiwan.}

\begin{document}

\begin{abstract}
    We construct a family of points on the Lagrangian cone of a partial flag bundle associated to a (possibly non-split) vector bundle from any Weyl-invariant $I$-function of a prequotient.
    This result can be seen as the nonabelian analogue of the mirror theorem for projective bundles \cite{IK:quantum}, and generalizes Oh's mirror theorem for split partial flag bundles \cite{Oh:split_flag_bundle}.
    
\end{abstract}

\maketitle

\tableofcontents

\section{Introduction}

Genus zero Gromov--Witten invariants of a smooth projective variety $X$ are given by virtual counts of rational curves on $X$.
They satisfy the expected axioms \cite{KM,Behrend}, and their generating function underlies very rich structures.
In particular, a (formal) Frobenius manifold can be constructed from the generating function, yielding the \emph{quantum cohomology ring} $\QH^*(X)$.

Givental formalism \cite{Givental:symplectic} introduces the \emph{Givental Lagrangian cone} $\cL_X$, a geometric object with special properties, which encodes the genus zero Gromov--Witten theory of $X$. 
An explicit formula describing (a family of) points on $\cL_X$ (historically, such formulas are often called  \emph{mirror theorems}, and we will follow this terminology) is a powerful tool for the study of $\QH^*(X)$.
We give some examples.
\begin{itemize}
    \item 
    In \cite{Iritani:convergence}, a mirror theorem for toric varieties was used to show that quantum cohomology of toric varieties is convergent.
    (This result was generalized to \emph{split} toric bundles \cite{Koto:convergence}, which are toric bundles obtained as GIT quotients by an algebraic torus of a direct sum of line bundles.)
    \item 
    In \cite{IK:quantum}, a mirror theorem for a projective bundle $\PP(V)$ (associated to a vector bundle $V\to B$) was established and was used to construct a decomposition $\QH^*(\PP(V))\cong \QH^*(B)^{\oplus\rank(V)}$. 
    \item 
    Dubrovin's conjecture \cite{Dubrovin:geometry} for Grassmannians was proved in \cite{Ueda:Dubrovin_Grassmannians}, and the Gamma conjectures were formulated and proved for Grassmannians in \cite{GGI}.
    In their proofs, the mirror theorem for Grassmannians from \cite{BCFK:two} plays an important role.
\end{itemize}

The purpose of this paper is to prove a mirror theorem for partial flag bundles (Theorem \ref{thm:mirror_Fl}).
Before stating it, we briefly review known mirror theorems related to ours.
Mirror theorems for partial flag varieties were established first in \cite{BCFK:two} (for Grassmannians), and then in \cite{BCFK:ab/nonab} (for general flag varieties, but only on the ``small" parameter space), \cite{CFKS} (on the ``big" parameter space).
These were extended to {\it split} partial flag bundles \cite{Oh:split_flag_bundle}.
On the other hand, the mirror theorem for a (possibly non-split) projective bundle $\PP(V)$ was given in \cite{IK:quantum} in terms of the ($\C^\times$-equivariant) $J$-function of $V$.
Theorem \ref{thm:mirror_Fl} generalizes these results to {\it all} (i.e., possibly non-split) partial flag bundles.

To simplify the notation, we state the result for Grassmann bundles in the introduction.
Let $B$ be a smooth projective variety, let $V\to B$ be a vector bundle whose dual is globally generated, and let $r$ be a positive integer less than $\rank(V)$.
Let $\Gr(r,V)\to B$ denote the Grassmann bundle whose fiber at $b\in B$ is the Grassmannian $\Gr(r,V_b)$ of $r$-dimensional subspaces of $V_b$.
The following is the main theorem of this paper (applied to the Grassmann bundle $\Gr(r,V)$).

\begin{theorem}[Special case of Theorem \ref{thm:mirror_Fl}]\label{thm:mirror_Gr}
    Let $B$, $V$, and $r$ be as above.
    Let $\T=(\C^\times)^r$ act on $V^{\oplus r}$ by scaling each component diagonally, and let $\W:=\mathfrak{S}_r$ act on $V^{\oplus r}$ by permuting the components.    
    Let $z\cdot I^{\lambda_\bullet}_{V^{\oplus r}}(x,z)$ be a $\W$-invariant function lying on the Givental cone of $V^{\oplus r}$ that depends on the $\T$-equivariant parameters $\lambda_1,\dots,\lambda_r$ polynomially.
    Define the $H^*(\Gr(r,V))$-valued function $I_{\Gr(r,V)}(t,x,z)$ as follows:
    \begin{multline*}
    I_{\Gr(r,V)}(t,x,z) = \sum_{k\geq0} q^ke^{\left(k+\frac{H}{z}\right)t} \sum_{ \substack{k_1,\dots,k_r\geq0; \\ \sum_{i=1}^rk_i=k } } \frac{ I^{H_\bullet+k_\bullet z}_{V^{\oplus r}}(x,z)\cdot(-1)^{k(r-1)} }{\prod_{i=1}^{r} \prod_{c=1}^{k_i} \prod_{ \substack{\delta\colon\mathrm{Chern\ roots} \\ \mathrm{of\ } V} } (H_i + \delta + cz)}	\\
    \cdot \prod_{1\leq i<j\leq r}\frac{H_i-H_j+(k_i-k_j)z}{H_i-H_j}
    \end{multline*}
    where $q$ denotes the Novikov variable for the line in fibers, $H_1,\dots,H_r$ denote the Chern roots of the dual of the tautological subbundle over $\Gr(r,V)$, while $H=H_1+\dots +H_r$, and $I^{H_\bullet+k_\bullet z}_{V^{\oplus r}}$ denotes the function $I^{\lambda_\bullet}_{V^{\oplus r}}$ with each $\lambda_i$ replaced with $H_i+k_iz$ for $1\leq i\leq r$.
    Then, $z\cdot I_{\Gr(r,V)}(t,x,z)$ lies in the Givental cone of $\Gr(r,V)$.
\end{theorem}

\begin{remark}
    To study the Givental cone $\cL_{\Gr(r,V)}$, the assumption that $V^\vee$ is globally generated is not necessary.
    In fact, by replacing $V$ with $V\otimes L$ where $L$ is a suitable anti-ample line bundle over $B$, we can always make $V$ satisfy the assumption without changing the bundle $\Gr(r,V)$.
\end{remark}

Theorem \ref{thm:mirror_Fl} recovers the $I$-function of projective bundles \cite[Theorem 3.3]{IK:quantum} and Oh's $I$-function of split partial flag bundles \cite[Theorem 1.1]{Oh:split_flag_bundle}; see Example \ref{ex:IK}, \ref{ex:Oh}. It can be viewed as an instance of the \emph{abelian/nonabelian correspondence} for partial flag bundles. We explain this, again restricting to the case of Grassmann bundles for simplicity. 

Let $\G$ be a complex reductive group and fix a maximal torus $\T$ in $\G$. Let $X$ be a quasi-projective variety with a $\G$-action and a $\G$-linearized ample line bundle, so that we have the GIT quotient $Y:=X/\!/\G$. Restricting both the action and the linearization to $\T$, we also have the \emph{abelian GIT quotient} $Y_{\operatorname{ab}}:=X/\!/\T$. For example, given a vector bundle $V\to B$, we may take $X=V^{\oplus r}\cong\operatorname{Hom}(\C^r,V)$ and $\G=\GL_r(\C)$, with the obvious action, linearized by the determinant character. In this case, the nonabelian quotient is precisely the Grassmann bundle $\Gr(r,V)=V^{\oplus r}/\!/\GL_r(\C)$, while the abelian quotient
$\Gr(r,V)_{\operatorname{ab}}=V^{\oplus r}/\!/(\C^\times)^r$ is the fiber product of $r$ copies of $\PP(V)$ over $B$.

Associated to the pair $(\G,\T)$ we have the Weyl group $\W$, the pair of Lie algebras $\mathfrak{g}:=Lie(\G), \mathfrak{t}:=Lie(\T)$, and the root system $\Phi$. The Weyl group acts on $Y_{\operatorname{ab}}$. The (split) adjoint $\T$-representation $\mathfrak{g}/\mathfrak{t}=\oplus_{\alpha\in\Phi}\mathfrak{g}_{\alpha}$ induces a split vector bundle $\oplus_{\alpha\in\Phi }L_{\alpha}$ on $Y_{\operatorname{ab}}$, which we will also call $\mathfrak{g}/\mathfrak{t}$ by abuse of notation. Informally, the \emph{abelian/nonabelian correspondence}, introduced and first studied in \cite{BCFK:ab/nonab,CFKS}, predicts that (under reasonable assumptions on the GIT set-up) the genus zero Gromov--Witten theory of the nonabelian quotient $Y$ is obtained by a kind of ``average over $\W$" of the genus zero Gromov--Witten theory of $Y_{\operatorname{ab}}$ \emph{twisted by the Lie algebra bundle $\mathfrak{g}/\mathfrak{t}$}. Precise conjectures resulting from this predicted relationship were stated in \cite{BCFK:ab/nonab,CFKS}, and some reformulations were given more recently in \cite{CLS}. Theorem \ref{thm:mirror_Gr} fits into this picture as follows.
A mirror theorem for the abelian quotient $\Gr(r,V)_{\operatorname{ab}}$ was proved in \cite[Theorem 5.1]{Koto:mirror}, giving an explicit point $z\cdot I_{\Gr(r,V)_{\operatorname{ab}}}(t,x,z)$. This point is $\W$-invariant and we may construct its ``hypergeometric modification" corresponding to the twisting bundle $\mathfrak{g}/\mathfrak{t}$. This modification gives precisely the function $z\cdot I_{\Gr(r,V)}(t,x,z)$, so the conclusion of Theorem \ref{thm:mirror_Gr} verifies that \cite[Conjecture 4.3]{BCFK:ab/nonab} and \cite[Conjecture 1.8]{CLS} hold true for the nonabelian quotient $\Gr(r,V)$. 

Let us explain briefly the steps involved in the proof of our main result, Theorem \ref{thm:mirror_Fl}. The first is a nonabelian version of an idea from \cite{IK:quantum} and \cite{Koto:mirror}. For a flag bundle $Fl(r_\bullet, V)$ associated to a vector bundle $V$ on $B$ with $V^\vee$ globally generated, we construct an embedding 
\[ 
\Fl(r_\bullet, V)\hookrightarrow B\times\Fl(r_\bullet, \C^N)
\]
into a larger {\em trivial} flag bundle. This realizes $\Fl(r_\bullet, V)$ as the zero locus of a regular section of a certain convex vector bundle, so (by \cite{KKP:functoriality}), the mirror theorem for $\Fl(r_\bullet, V)$ is reduced to proving Lemma \ref{lem:twisted_function}, which is a mirror theorem for {\em twisted} Gromov--Witten theory of the product $B\times\Fl(r_\bullet, \C^N)$.

The second step (see $\S 6.1$) further reduces Lemma \ref{lem:twisted_function} to an analogous mirror theorem for twisted theory of a trivial toric bundle. To accomplish this reduction, we establish the following abelian/nonabelian correspondence for products.

\begin{theorem}[Non-equivariant version of Theorem \ref{thm:abelian/nonabelian_correspondence}]\label{thm:intro_ab/nonab}
    Let $\FlG$ be a partial flag variety (viewed as a GIT quotient), and let $\FlT$ denote the corresponding abelian quotient (a smooth, projective toric variety).
    Then, the $\gt$-modification (defined in Section~\ref{subsec:statement_of_ab/nonab}) sends an arbitrary $\W$-invariant point of the Givental cone of $B\times\FlT$ satisfying the Divisor Equation~\eqref{eqn:Divisor_Equation} into the Givental cone of $B\times\FlG$.
\end{theorem}

In fact, we prove a more general torus-equivariant version, see Theorem~\ref{thm:abelian/nonabelian_correspondence}.
The key ingredients of the proof are a Givental/Brown-type characterization of points on the torus-equivariant Givental cone of $B\times\FlG$ (Theorem \ref{thm:characterization}), and the abelian/nonabelian correspondence for partial flag varieties \cite{CFKS}.

In the third and last step, we prove in $\S 6.4$ the required mirror theorem for the twisted theory of $B\times\FlT$ by the methods of \cite{IK:quantum} and \cite{Koto:mirror}, which are based on an analysis of the action of certain differential operators on the Givental cone, especially the quantum Riemann-Roch operators \cite{CG:quantum}.

Our main theorem (Theorem \ref{thm:mirror_Fl}) has an application.
It provides a nonabelian counterpart of the Fourier transforms introduced in \cite{IK:quantum,Iritani:blowups} and enables one to study quantum cohomology of general partial flag bundles, e.g., to prove a decomposition of quantum cohomology/quantum $D$-module analogous to those in \cite{IK:quantum,Iritani:blowups}. These will be discussed in the follow-up paper \cite{Koto:GIT}. Another application is the extension of Theorem \ref{thm:intro_ab/nonab} (the full abelian/nonabelian correspondence) from trivial flag bundles to arbitrary flag bundles. This will be also pursued in \cite{Koto:GIT}, where it will be shown that by choosing appropriate $I$-functions for the vector bundle $V$, the formula of Theorem \ref{thm:mirror_Fl} in fact describes the entire Langrangian cone.

A brief description of the rest of the paper is as follows.
In Section \ref{sec:GW}, we review genus zero Gromov--Witten theory and Givental formalism, and prepare some lemmas for later use. In Section \ref{sec:flag_bundles} we review flag bundles, including their construction as GIT quotients, and explain their embeddings into trivial flag bundles mentioned above.
In Section \ref{sec:main_theorem}, we prove the main theorem (Theorem \ref{thm:mirror_Fl}) using Lemma \ref{lem:twisted_function}.
This is proved in Section \ref{sec:proof_of_Lemma}, and the abelian/nonabelian correspondence (Theorem \ref{thm:intro_ab/nonab}) needed for its proof is given in Section \ref{sec:ab/nonab}.

\phantom{A}

\noindent
\textbf{Acknowledgements.}
The second named author is grateful to Hiroshi Iritani for their joint research \cite{IK:quantum}, which led to many of the ideas in this paper. Both authors thank Hiroshi Iritani, Jeongseok Oh, and Hsian-Hua Tseng for helpful comments.

\section{Genus zero Gromov--Witten theory}\label{sec:GW}
We review genus zero Gromov--Witten theory, and describe the Givental cone and its variants.

\subsection{Gromov--Witten invariant}
In this subsection, we explain Gromov--Witten invariants and its variants.

\subsubsection{Gromov--Witten invariant}
Let $X$ be a smooth projective variety.
We denote by $\NE(X)$ the monoid in $H_2(X;\Z)$ consisting of the effective curve classes.

For any non-negative integer $n$ and any effective curve class $d\in\NE(X)$, let $\ovcM_{0,n}(X,d)$ denote the moduli stack of degree $d$ stable maps from genus zero prestable curves to $X$ with $n$ markings.
For any classes $\alpha_1,\dots,\alpha_n\in H^*(X;\Q)$ and any integers $k_1,\dots,k_n\in\Z_{\geq0}$, we define a (\emph{genus zero}) \emph{descendant Gromov--Witten invariant} by
\[
\corr{\alpha_1\psi^{k_1},\dots,\alpha_n\psi^{k_n}}^X_{0,n,d} := \int_{[\ovcM_{0,n}(X,d)]^{\vir}}\prod^n_{i=1}\ev^*_i(\alpha_i)\psi_i^{k_i}\in\Q
\]
where $[\ovcM_{0,n}(X,d)]^\vir\in H_*(\ovcM_{0,n}(X,d);\Q)$ is the virtual fundamental class \cite{Behrend}, and $\psi_i\in H^2(\ovcM_{0,n}(X,d))$ ($1\leq i\leq n$) denotes the first Chern class of the $i$-th universal cotangent line bundle over $\ovcM_{0,n}(X,d)$.

\subsubsection{torus-equivariant Gromov--Witten invariant}

Let $\T$ be an algebraic torus, and let $X$ be a smooth quasi-projective $\T$-variety with projective fixed loci $X^\T(\neq\emptyset)$.
Note that the $\T$-action on $X$ induces a $\T$-action on the moduli spaces $\ovcM_{0,n}(X,d)$.

We define a (\emph{genus zero}) \emph{$\T$-equivariant descendant Gromov--Witten invariant} via the virtual localization formula \cite{GP}:
\[
\corr{\alpha_1\psi^{k_1},\dots,\alpha_n\psi^{k_n}}^{X,\T}_{0,n,d} := \int_{[\ovcM_{0,n}(X,d)]^{\vir}}^\T\prod^n_{i=1}\ev^*_i(\alpha_i)\psi_i^{k_i}\in H^*_\T(\pt;\Q)_\loc\cong\Q({\mu})
\]
where $\alpha_i\in H^*_\T(X)$, $k_i\geq0$ and $\mu=(\mu_1,\dots ,\mu_{\operatorname{rk}\T})$ denotes the equivariant parameters for $\T$.
If $X$ is projective, this quantity can be defined without localization and lies in $H^*_\T(\pt;\Q)$.

\subsubsection{twisted Gromov--Witten invariant}
Let $X$ be a smooth projective variety.
Fix \emph{twist data $(W,\bc)$ for $X$}, which is a pair of a $\T$-vector bundle $W$ over $X$ and a characteristic class $\bc\colon K(X)\to \C(\mu)$.
In this paper, we only deal with the case that $\bc$ is the $\T$-equivariant Euler characteristic class $e_\T=\prod\rho(\mu)$ or its inverse (here $\rho(\mu)$ are the $\T$-equivariant Chern roots of $W$).
For any $n\geq0$ and $d\in\NE(X)$, set $W_{0,n,d}:=\R\pi_*f^*W\in K_\T(\ovcM_{0,n}(X,d))$ where $\pi\colon \ovcM_{0,n+1}(X,d)\to\ovcM_{0,n}(X,d)$ denotes the universal family and $f\colon \ovcM_{0,n+1}(X,d)\to X$ denotes the universal stable map.
For $\alpha_i\in H^*_\T(X)$ and $k_i\geq0$ $(1\leq i\leq n)$, we define a (\emph{genus zero}) \emph{$(W,\bc)$-twisted descendant Gromov--Witten invariant} as 
\[
\corr{\alpha_1\psi^{k_1},\dots,\alpha_n\psi^{k_n}}^{X,(W,\bc)}_{0,n,d} := \int_{[\ovcM_{0,n}(X,d)]^{\vir}}\left[\prod^n_{i=1}\ev^*_i(\alpha_i)\psi_i^{k_i}\right]\cup\bc(W_{0,n,d})\in\Q(\mu).
\]

\begin{remark}\label{rmk:GW_bundle}
    If $V\to B$ is a $\T$-vector bundle whose dual is globally generated and we assume in addition that the $\T$-fixed locus in (the total space of) $V$ is precisely the zero section, then $(V,e_\T^{-1})$-twisted Gromov--Witten invariants coincide with $\T$-equivariant Gromov--Witten invariants of $V$.
    More precisely, for any $\alpha_i\in H^*_\T(V)\cong H^*(B)[\mu]$ and $k_i\geq0$,
    \[
    \corr{\alpha_1\psi^{k_1},\dots,\alpha_n\psi^{k_n}}^{V,\T}_{0,n,d} = \corr{\alpha_1\psi^{k_1},\dots,\alpha_n\psi^{k_n}}^{B,(V,e_\T^{-1})}_{0,n,d}.
    \]
\end{remark}

\subsection{Quantum cohomology}
Let $X$ be a smooth projective variety.
We fix a homogeneous $\C$-basis $\{\phi_i\}_{i\in I}$ of $H^*(X,\C)$, and write $\{\tau^i\}_{i\in I}$ the associated dual coordinate system on $H^*(X,\C)$.
Let $Q=Q_X$ denote a formal variable for the monoid $\NE(X)$ (called the \emph{Novikov variable}), and let $\C[\![Q]\!]:=\C[\![\NE(X)]\!]$ (called the \emph{Novikov ring}) be the completion of the monoid ring $\C[\NE(X)]$.
For any (graded) module $R$, we set $R[\![Q,\tau]\!]:=R[\![Q]\!][\![\{\tau^i\}_{i\in I}]\!]$, which is regarded as a graded module with the following grading: 
\[
\deg\tau^i=2-\deg\phi_i \quad (i\in I), \qquad \deg Q^d=c_1(X)\cdot d \quad (d\in\NE(X)).
\]

We write $\corr{\cdot,\cdot}^X$ for the Poincar\'{e} pairing on $H^*(X,\C[\![Q,\tau]\!])$.
For any $\alpha,\beta\in H^*(X)[\![Q,\tau]\!]$, define the \emph{quantum product} of $\alpha$ and $\beta$ to be the element $\alpha\star_{\tau}\beta\in H^*(X)[\![Q,\tau]\!]$ satisfying that, for any $\gamma\in H^*(X)$,
\[
\corr{\alpha\star_{\tau}\beta,\gamma}^X = \sum_{n\geq0} \sum_{d\in\NE(X)} \frac{Q^d}{n!} \corr{\alpha,\beta,\gamma,\tau,\dots,\tau}^X_{0,n+3,d} \in \C[\![Q,\tau]\!]
\]
where $\tau=\sum_{i\in I}\tau^i\phi_i$.
Note that $\tau$ is homogeneous of degree two.
Since the Poincar\'{e} pairing is nondegenerate, $\alpha\star_{\tau}\beta$ is uniquely determined.
We define the \emph{quantum cohomology} of $X$ to be the graded ring $\QH^*(X):=(H^*(X)[\![Q,\tau]\!],\star_{\tau})$.

We define the \emph{fundamental solution} \cite{Givental:equivariant,Pandharipande:rational,CCIT:hodge} $M_X(\tau)\in \End(H^*(X))[z^{-1}][\![Q,\tau]\!]$ as follows:
\begin{equation}\label{eqn:fundamental_solution}
\corr{M_X(\tau)\alpha,\beta}^X = \corr{\alpha,\beta}^X + \sum_{ \substack{n\geq0,d\in\NE(X) \\ (n,d)\neq(0,0)} } \frac{Q^d}{n!} \corr{\alpha,\bt(\psi),\dots,\bt(\psi),\frac{\beta}{z-\psi}}^X_{0,n+2,d}
\end{equation}
where $1/(z-\psi)$ is expanded as $\sum_{k\geq0}z^{-k-1}\psi^k$.
The matrix $M_X(\tau)$ is a solution of the \emph{quantum connection} $\nabla^X_{z\partial_{\tau^i}}=z\partial_{\tau^i}+(\phi_i\star_\tau\cdot)$.
In particular, $M_X(\tau)$ is the matrix that satisfies 
\begin{equation}\label{fundamental solution}
M_X(\tau)\circ\nabla^X_{z\partial_{\tau^i}}=z\partial_{\tau^i}\circ M_X(\tau)    
\end{equation} 
for each $i\in I$, with the initial condition $M_X(\tau)|_{Q=0}=e^{\tau/z}$.

Define the \emph{$J$-function} of $X$ as follows.
\begin{equation}\label{eqn:J-function}
    J_X(\tau,z):= M_X(\tau)1 = 1 + \frac{\tau}{z} + \sum_{ \substack{n\geq0,d\in\NE(X) \\ (n,d)\neq(0,0),(1,0)} } \sum_{i\in I} \frac{Q^d\phi_i}{n!} \corr{\tau,\dots,\tau,\frac{\phi^i}{z(z-\psi)}}^X_{0,n+1,d}
\end{equation}
where $\{\phi^i\}_{i\in I}$ is the dual basis of $\{\phi_i\}_{i\in I}$ with respect to the Poincar\'{e} pairing.

\subsection{Givental cone}\label{subsec:Givental_cone}
The \emph{Givental} (\emph{Lagrangian}) \emph{cone} plays a significant role in Givental formalism \cite{Givental:quantization,Givental:symplectic}, which encodes genus zero Gromov--Witten theory into the symplectic geometry of the \emph{Givental space}, an infinite-dimensional symplectic vector space.

In this section, we will introduce three kinds of Lagrangian cones: a Lagrangian cone of smooth projective varieties (Section \ref{subsubsec:Givental_cone}), a $\T$-equivariant Givental cone of smooth quasi-projective varieties (Section \ref{subsubsec:equivariant_Givental_cone_qproj}), and a twisted Givental cone (Section \ref{subsubsec:twisted_Givental_cone}).

Throughout this section, we assume that $X$ is a smooth quasi-projective variety.
Fix a homogeneous basis $\{\phi_i\}_{i\in I}$ of $H^*(X,\C)$, and let $\{\tau^i\}_{i\in I}$ denote the associated coordinate and $\{\phi^i\}_{i\in I}$ denote its dual basis with respect to the Poincar\'{e} pairing.
Finally, we set $Q:=Q_X$. 

\subsubsection{Givental cone for smooth projective varieties}\label{subsubsec:Givental_cone}

The \emph{Givental space} $(\cH_X,\Omega_X)$ of $X$ is the infinite-dimensional free $\C[\![Q]\!]$-module $\cH_X = H^*(X)[z,z^{-1}][\![Q]\!]$ equipped with the symplectic form $\Omega_X(f(z),g(z)) = -\Res_{z=\infty}\corr{f(-z),g(z)}^Xdz$.
$\cH_X$ splits canonically into two maximally $\Omega_X$-isotropic subspaces $\cH_+$ and $\cH_{-}$:
\[
\cH_+ = H^*(X)[z][\![Q]\!], \quad \cH_- = z^{-1}H^*(X)[z^{-1}][\![Q]\!].
\]
Since $\cH_-$ is identified with the dual of $\cH_+$ via $\Omega_X$, we have a canonical identification $\cH_X\cong T^*\cH_+$. 

The \emph{Givental Lagrangian cone} $\cL_X \subset \cH_X$ \cite{Givental:symplectic} is defined as the graph of the differential of the genus zero descendant Gromov--Witten potential (via $\cH_X\cong T^*\cH_+$).
We give an explicit description of the points on $\cL_X$.
Let $x=(x_1,x_2,\dots)$ be a set of (possibly infinitely many) formal parameters.
Then, a $\C[\![Q,x]\!]$-valued point on $\cL_X$ is of the form
\begin{equation}\label{eqn:point_on_cone}
\bone z + \bt(z) + \sum_{ \substack{n\geq0,d\in\NE(X) \\ (n,d)\neq(0,0),(1,0)} } \sum_{i\in I} \frac{Q^d\phi_i}{n!} \corr{\bt(-\psi),\dots,\bt(-\psi),\frac{\phi^i}{z-\psi}}^X_{0,n+1,d}
\end{equation}
for some $\bt(z)\in H^*(X)[z][\![Q,x]\!]$ satisfying $\bt(z)|_{Q=x=0}=0$.
Here, we follow the sign convention in \cite{IK:quantum,Iritani:blowups}, and $1/(z-\psi)$ is expanded as $\sum_{k\geq0}z^{-k-1}\psi^k$.
Since $\psi$ is nilpotent, this gives a well-defined element of $\cH_X(\C[\![Q,x]\!])=H^*(X)[z,z^{-1}][\![Q,x]\!]$.

\begin{remark}\label{rmk:R-valued_point}
    More generally, we can define a $R$-valued point on $\cL_X$ for arbitrary complete local $\C[\![Q]\!]$-algebra $(R,\mathfrak{m})$ by the same formula as \eqref{eqn:point_on_cone} with $\bt(z)\in H^*(X)[z][\![Q]\!]\widehat{\otimes}_{\C[\![Q]\!]}R$, where $\widehat{\otimes}$ denotes the completed tensor product with respect to $\mathfrak{m}$, such that $\bt(z)=0$ modulo $\mathfrak{m}$; see \cite[Appendix B]{CCIT:computing}.
\end{remark}

Axioms of Gromov--Witten invariants imply that $\cL_X\subset\cH_X$ is \emph{overruled} \cite{Givental:symplectic}.
In particular, any $\C[\![Q,x]\!]$-valued point on $\cL_X$ can be written as $zM_X(\tau(x))f$ for some $\tau(x)\in H^*(X)[\![Q,x]\!]$ and $f\in H^*(X)[z][\![Q,x]\!]$ satisfying that $\tau(x)|_{Q=x=0}=0$ and $f|_{Q=x=0}=1$.

\subsubsection{$\T$-equivariant Givental cone of smooth quasi-projective varieties}\label{subsubsec:equivariant_Givental_cone_qproj}
Let $\T$ be an algebraic torus, and let $X$ be a smooth quasi-projective $\T$-variety whose $\T$-fixed set is projective.
The \emph{$\T$-equivariant Givental Lagrangian cone} $\cL_X$ is defined by replacing spaces/formulas used to define $\cL_X$ with their $\T$-equivariant counterparts.
We explain them in a list.
\begin{itemize}
    \item
    $\cH_{X,\T} := H^*_\T(X)_\loc(\!(z^{-1})\!)[\![Q]\!]$ where $H^*_\T(X)_\loc:=H^*_\T(X)\otimes_{H^*_\T(\pt)}\Frac(H^*_\T(\pt))$. (In this paper, we use $A((y))$ to denote the ring of formal Laurent series in the variable $y$, with coefficients in the $\C$-algebra $A$.) 
    \item 
    $\Omega_X(f(z),g(z)) = -\Res_{z=\infty}\corr{f(-z),g(z)}^{X,\T}dz$ for $f,g\in\cH_{X,\T}$, where $\corr{\cdot,\cdot}^{X,\T}$ denotes the $H^*_\T(\pt)_\loc$-bilinear pairing on $H^*_\T(X)_\loc$ defined by the localization formula.
    \item 
    $\cH_+:=H^*_\T(X)_\loc[z][\![Q]\!]$ and $\cH_-:=z^{-1}H^*_\T(X)_\loc[\![z^{-1}]\!][\![Q]\!]$.
\end{itemize}
For any formal variables $x=(x_1,x_2,\dots)$, a $H^*_\T(\pt)_\loc[\![Q,x]\!]$-valued point on $\cL_{X,\T}$ is of the form
\[
\bone z + \bt(z) + \sum_{ \substack{n\geq0,d\in\NE(X) \\ (n,d)\neq(0,0),(1,0)} } \sum_{i\in I} \frac{Q^d\phi_i}{n!} \corr{\bt(-\psi),\dots,\bt(-\psi),\frac{\phi^i}{z-\psi}}^{X,\T}_{0,n+1,d}
\]
for some $\bt(z)\in H^*_\T(X)_\loc[z][\![Q,x]\!]$ such that $\bt(z)|_{Q=x=0}=0$.
As in the non-equivariant theory, for any $H^*_\T(\pt)_\loc[\![Q,x]\!]$-valued point $z\cdot I$ of $\cL_{X,\T}$, there exist 
\[
\tau(x)\in H^*_\T(X)_\loc[\![Q,x]\!],\quad  f\in H^*_\T(X)_\loc[z][\![Q,x]\!]
\]
satisfying $\tau(x)|_{Q=x=0}=0$ and $f|_{Q=x=0}=1$ such that $I=M_{X,\T}(\tau(x))f$.
Here, $M_{X,\T}(\tau)$ represents the $\T$-equivariant fundamental solution of $X$, and is defined by the same formula \eqref{eqn:fundamental_solution}, with Gromov--Witten invariants in \eqref{eqn:fundamental_solution} replaced by $\T$-equivariant Gromov--Witten invariants.

\subsubsection{twisted Givental cone}\label{subsubsec:twisted_Givental_cone}
Let $\T$ be an algebraic torus, and let $E\to X$ be a vector bundle over a smooth projective variety $X$ with a fiberwise $\T$-action.
We assume that $E^\T=X$ (embedded as the zero section of $E$) and $\bc$ is either the $\T$-equivariant Euler class or its inverse.
Again, we define the cone for $(E,\bc)$-twisted Gromov--Witten theory of $X$ by the same procedure.

Namely, the Givental space for the twisted theory is
\[
\cH_{X,(E,\bc)} = H^*_\T(X)_\loc(\!(z)\!)[\![Q]\!], \quad \cH_+ = H^*_\T(X)_\loc[\![z,Q]\!].
\]
For any formal variables $x=(x_1,x_2,\dots)$, a $H^*_\T(\pt)_\loc[\![Q,x]\!]$-valued point on $\cL_{X,(E,\bc)}$ is of the form
\[
\bone z + \bt(z) + \sum_{ \substack{n\geq0,d\in\NE(X) \\ (n,d)\neq(0,0),(1,0)} } \sum_{i\in I} \frac{Q^d\phi_i}{n!} \corr{\bt(-\psi),\dots,\bt(-\psi),\frac{\bc(E)^{-1}\phi^i}{z-\psi}}^{X,(E,\bc)}_{0,n+1,d}.
\]
with $\bt(z)\in H^*_\T(X)_\loc[\![z,Q,x]\!]$ such that $\bt(z)|_{Q=x=0}=0$.

A point on the Givental cone of a subvariety can sometimes be obtained from a known point on the Euler-twisted Givental cone of the ambient variety. 
More precisely, the following standard lemma (a consequence of the functoriality of virtual classes in \cite{KKP:functoriality})  holds.

\begin{lemma}\label{lem:I-function_of_subvarieties}
    Let $X$ be a smooth projective variety, and let $E\to X$ be a convex vector bundle, i.e. $H^1(C,f^*E)=0$ for any genus zero stable map $f\colon C\to X$. Let $Z\stackrel{\iota}{\subset} X$ be the zero-scheme of a regular section of $E$. Consider the $\C^\times$-action on $E$ by scaling, with equivariant parameter $\mu$, and write $e_\mu$ for the equivariant Euler class $e_{\C^\times}$.
    Let $f$ be a $\C(\mu)[\![Q_X,x]\!]$-valued point of $\cL_{X,(E,e_\mu)}$.
    If the non-equivariant limit of the pullback $\iota^*f$ exists, then the limit $\lim_{\mu\to0}\iota^*f$ is a $\C[\![Q_X,x]\!]$-valued point of $\cL_{Z}$.
    Here, $\C[\![Q_X,x]\!]$ is regarded as a $\C[\![Q_Z]\!]$-algebra via the map $Q_Z^d\mapsto Q_X^{\iota_*d}$.
\end{lemma}

\begin{remark}
    When $E$ is a convex vector bundle, then the $(E,e_\mu)$-twisted invariants are polynomials in $\mu$ as $E_{0,n,d}$ can be realized as an actual vector bundle, and hence we can consider $(E,e_\mu)$-twisted theory without localizing with respect to $\mu$.
\end{remark}

\subsection{$\T$-equivariant Givental cone of vector bundles}\label{subsec:equivariant_Givental_cone_vector_bundles}
Let $\T$ be an algebraic torus, and let $V\to X$ be a vector bundle over a smooth projective variety $X$ with a fiberwise $\T$-action.
We assume that $V^\vee$ is globally generated and $V^\T=X$.
We have already introduced the $\T$-equivariant Givental space of certain quasi-projective $\T$-varieties (Section \ref{subsubsec:equivariant_Givental_cone_qproj}), including the bundle $V$.
In this subsection, we introduce a smaller variant for $V$, which will be denoted by $\cH_{V,\T}^\pol$.
We will use $\cH_{V,\T}^\pol$ (and its subset $\cL_{V,\T}$) to formulate our main theorem.

We follow almost the same procedure as in the previous subsection, with some modifications.
We identify $\NE(V)$ and $\NE(X)$ and write $Q$ for the Novikov variable of $V$.
Define the Givental space $\cH_{V,\T}^\pol$ and its subspace $\cH_+$ as follows:
\[
\cH_{V,\T}^\pol=H^*_\T(V)[z,z^{-1}][\![Q]\!],\quad\cH_+^\pol=H^*_\T(V)[z][\![Q]\!].
\]
We define the Givental cone $\cL_{V,\T}\subset\cH_{V,\T}^\pol$ as follows: for any formal variables $x=(x_1,x_2,\dots)$, the set $\cL_{V,\T}(H^*_\T(\pt)[\![Q,x]\!])$ consists of the points of the form
\begin{equation}\label{eqn:point_on_cone_vector_bundle}
\bone z + \bt(z) + \sum_{ \substack{n\geq0,d\in\NE(X) \\ (n,d)\neq(0,0),(1,0)} } \sum_{i\in I} \frac{Q^d\phi_i}{n!} \corr{\bt(-\psi),\dots,\bt(-\psi),\frac{e_\lambda(V)\phi^i}{z-\psi}}^{V,\T}_{0,n+1,d}
\end{equation}
with $\bt(z)\in H^*_\T(V)[z][\![Q,x]\!]$ such that $\bt(z)|_{Q=x=0}=0$.
The set $\{e_\lambda(V)\phi^i\}_{i\in I}$ is the dual basis of $\{\phi_i\}_{i\in I}$ with respect to the Poincar\'{e} pairing $\corr{\phi_1,\phi_2}^V = \int_X^\T\phi_1\cup\phi_2\cup e_\lambda(V)^{-1}$.

The function \eqref{eqn:point_on_cone_vector_bundle} is indeed a point of $\cH_{V,\T}^\pol$.
In fact, the class $\psi_1$ for $\ovcM_{0,n}(V,d)$ is nilpotent since it matches the pullback of $\psi_1$ for $\ovcM_{0,n}(B,d)$ as images of $\T$-fixed stable maps to $V$ are included in the zero section $B\subset V$.

\subsection{Quantum Riemann-Roch theorem}
We recall the Quantum Riemann-Roch Theorem of Coates-Givental, \cite{CG:quantum}, which relates a twisted Givental cone with the untwisted one.

Let $X$ be a smooth projective variety, and let $E\to X$ be a vector bundle.
For a formal variable $\lambda$, we define the operator $\Delta_E^\lambda$ acting on $\cH_X$ as
\[
\Delta_E^\lambda f := \left. \left( \exp\left[ \sum_{l,m\geq0} s_{l+m-1}(\lambda) \cdot \frac{B_m}{m!} \cdot \mathrm{ch}_l(E) \cdot (-z)^{m-1} \right] f \right) \right|_{Q_X^d\to Q_X^d\lambda^{-d\cdot c_1(E)}},
\]
where $B_m$ is the Bernoulli number given by the formula $x/(e^x-1)=\sum_{m=0}^\infty(B_m/m!)x^{(m)}$, and $s_k(\lambda)$ ($k\geq-1$) is defined as
\[
s_k(\lambda) = 
\begin{cases}
    0   &\text{for}\quad k=-1,0,  \\
    (k-1)!(-\lambda)^{-k} &\text{for}\quad k\geq1.
\end{cases}
\]
Note that the characteristic class $\lambda^{-\rank(\cdot)}\exp(\sum_{k\geq0}s_k(\lambda)\cdot\mathrm{ch}_k(\cdot))$ matches $e_\lambda(\cdot)^{-1}$ (the {\em inverse} of the $\C^\times$-equivariant Euler class, where $\lambda$ is regarded as the equivariant parameter for the $\C^\times$-action on $E$ by multiplication on the fibers).
The following theorem follows from the quantum Riemann-Roch theorem \cite[Corollary 4]{CG:quantum}. 

\begin{theorem}\label{thm:QRR}
    Let $\T$ be an algebraic torus, and let $E$ be a $\T$-vector bundle over $X$ with a decomposition $E=\bigoplus_{\alpha\in\Hom(\T,\C^\times)}E_\alpha$.
    Then, we have $\cL_{X,(E,e_\T^{-1})}=(\prod_{\alpha\in\Hom(\T,\C^\times)}\Delta_{E_\alpha}^\alpha)\cL_X$.
\end{theorem}

\subsection{Differential operators acting on Givental cones}
In this subsection, we introduce known results on differential operators acting on Givental cones.
It is important to study them because they provide (genuinely) new points on $\cL_X$ (or $\cL_{X,\T}$) from known points on it.

The following lemma claims that a certain differential operator acts on the Givental cone; see e.g. \cite[\S8]{CG:quantum}, \cite[Lemma 2.7]{IK:quantum}.

\begin{lemma}\label{lem:differential_operator_preserves_cone}
    Let $\T$ be an algebraic torus, and let $X$ be a smooth projective $\T$-variety.
    Let $x=(x_1,x_2,\dots)$ and $y=(y_1,y_2,\dots)$ be formal variables.
    Let 
    \[
    D(x,y,z\partial_x,z)\in \C[z][\![x]\!]\langle z\partial_x \rangle [\![Q_X,y]\!]
    \]
    be a differential operator such that $D|_{Q_X=y=0}=0$.
    Then, the differential operator $e^{D/z}$ preserves $\cL_X$ and $\cL_{X,\T}$ respectively.
\end{lemma}

The following lemma asserts that an element of the Givental cone with sufficiently many parameters can be obtained by applying a differential operator to the $J$-function.

\begin{lemma}[{\cite[Lemma 2.8]{IK:quantum}}]\label{lem:moving_points}
    Let $B$ be a smooth projective variety, let $\{\phi_i\}_{i\in I}$ be a basis of $H^*(B,\C)$, and let $\{\tau^i\}_{i\in I}$ be the dual coordinates on $H^*(B,\C)$.
    Let $z\cdot I(\tau,x,z)$ be an $\C[\![Q_B,\tau,x]\!]$-valued point on $\cL_B$ such that
    \[
    z \cdot I(\tau,x,z)|_{Q_B=x=0} = z + \sum_{i\in I}\tau^i\phi_i + O(z^{-1}).
    \]
    Then, there exists a differential operator $D\in\sum_{i\in I}\C[z][\![Q_B,\tau,x]\!]z\partial_{\tau^i}$ such that $e^{D/z}J_B(\tau,z)=I(\tau,x,z)$.
\end{lemma}

\section{Flag bundles, associated abelian quotients, and Weyl-invariant Givental cones}\label{sec:flag_bundles}
In this section we recall the construction (which we consider folklore) of partial flag bundles as GIT quotients of 
certain vector bundles by a product of general linear groups. The associated abelian quotients are obtained as GIT quotients by a maximal torus. These abelian quotients carry an action of the Weyl group, leading to the notion of \emph{Weyl-invariant Givental cone} of an abelian quotient, studied in \cite{CLS}. In the last subsection we discuss certain natural embeddings of flag bundles which will be used in the proof of the mirror theorem.  

\subsection{Partial flag bundles as GIT quotients}\label{subsection: flag_bundles_GIT}
Let $B$ be a smooth projective variety, and let $V\to B$ be a vector bundle.
Let $r_\bullet=\{r_1<\dots<r_l\}$ be a subset of $\{1,2,\dots,\rank(V)-1=n-1\}$.
Set
\[
\widetilde{V}=\bigoplus_{m=1}^{l-1}\Hom(\cO_B^{\oplus r_m},\cO_B^{\oplus r_{m+1}})\oplus\Hom(\cO_B^{\oplus r_l},V), \quad \G=\prod_{m=1}^{l}\GL_{r_m}(\C). 
\]
Let $\G$ act on $\widetilde{V}$ as follows: for $(F_1,\dots,F_l)\in\widetilde{V}$ and $(g_1,\dots,g_l)\in\G$, we define
\[
(g_1,\dots,g_l)\cdot(F_1,\dots,F_l) = (g_2^{-1}\circ F_1\circ g_1,g_3^{-1}\circ F_2\circ g_2,\dots,g_l^{-1}\circ F_{l-1}\circ g_{l-1},F_l\circ g_l).
\]
Let $\T\subset\G$ be the maximal torus consisting of all diagonal matrices.
Let $\theta\in\chi(\G)$ be the character $\theta\colon\G\to\C^\times$ given by the product of all determinants of each factor $\GL_{r_m}(\C)$, and let $L_\theta\cong\cO_{\widetilde{V}}$ be the trivial line bundle on $\widetilde{V}$, with $\G$-linearization induced by $\theta$. Viewing $\theta$ as a character of $\T$ (by restricting) makes $L_\theta$ into a $\T$-linearized line bundle as well.
We define the \emph{nonabelian quotient} $\widetilde{V}/\!/\G$ and the \emph{abelian quotient} $\widetilde{V}/\!/\T$ to be the following relative GIT quotients:
\[
\widetilde{V}/\!/\!_{\theta}\G := \underline{\mathrm{Proj}}_B \left( \bigoplus_{k\geq0} \pi_*\left(L_\theta^{\otimes k}\right)^\G \right),\qquad \widetilde{V}/\!/\!_{\theta}\T := \underline{\mathrm{Proj}}_B \left( \bigoplus_{k\geq0} \pi_*\left(L_\theta^{\otimes k}\right)^\T \right)
\]
where $\pi\colon\widetilde{V}\to B$ denotes the canonical projection. 

In either case, there are no strictly semistable points and the groups $\G$ and $\T$ act freely on their respective stable loci with respect to $L_\theta$.
For the $L_\theta$-linearized $\G$-action, a point $(F_1,\dots,F_l)\in\widetilde{V}$ is stable if and only if each $F_i$ is injective. It follows that
the nonabelian quotient is naturally identified with the partial flag bundle over $B$ (whose fiber at $b\in B$ is the partial flag variety $\Fl(r_\bullet,V_b)$). The abelian quotient is a smooth projective variety which is a toric bundle over $B$. 

If $M$ is a line bundle on $B$ and we replace $V$ by $V\otimes M$ in the above construction, then there are canonical isomorphisms (of $B$-schemes) $\widetilde{V}/\!/\!_{\theta}\G\cong \widetilde{V\otimes M}/\!/\!_{\theta}\G$ and $\widetilde{V}/\!/\!_{\theta}\T\cong \widetilde{V\otimes M}/\!/\!_{\theta}\T$. In particular, by fixing an ample line bundle $\mathcal{O}_B(1)$ on $B$ and taking $M\cong\mathcal{O}_B(-m)$ for sufficiently large $m$, we may (and will from now on) assume that the dual of $V$ is generated by global sections.

We introduce notations for abelian/nonabelian quotients.

\begin{notation}\label{notation:flag_bundle}
    Let $V\to B$ and $r_\bullet$ be as above.
    \begin{enumerate}
        \item 
        We write $\Fl(r_\bullet,V)$ (resp. $\Fl(r_\bullet,V)_{\operatorname{ab}}$) for the nonabelian quotient $\widetilde{V}/\!/\G$ (resp. the abelian quotient $\widetilde{V}/\!/\T$).
        \item 
        A $\G$-linearized vector bundle $E$ over $\widetilde{V}$ induces a vector bundle over $\Fl(r_\bullet,V)$, which we denote $E_\G$, and also (by viewing it as a $\T$-linearized bundle) a vector bundle over the abelian quotient $\Fl(r_\bullet,V)_{\operatorname{ab}}$, denoted $E_\T$. Each of these is obtained by restricting $E$ to the corresponding stable locus and then descending to the quotient, which is always possible as all stable points have trivial stabilizers. In this paper, we will apply the above construction to bundles of the form $E=\pi^*F\otimes (\mathcal{O}_{\widetilde{V}}\otimes R)$ where $F$ is vector bundle on $B$ (with trivial $\G$-action) and $R$ is a finite-dimensional $\G$-representation.
        \item 
        We denote by $S_1\subset S_2\subset\cdots\subset S_l=:S\subset\pi^*V$ the tautological flag of subbundles over $\Fl(r_\bullet,V)$.
        Note that each $S_m$ can be realized as $(E_{r_m})_\G$ where $E_{r_m}=\mathcal{O}_{\widetilde{V}}\otimes\C^{r_m}$ is the trivial bundle with fiber the dual of the standard representation of $\GL_{r_m}(\C)$.
        We denote by $(S_m)_\T$ the bundle over $\Fl(r_\bullet,V)_{\operatorname{ab}}$ induced by $E_{r_m}$.
        We set $S_\T:=(S_l)_\T$. 
        \item 
        The \emph{Weyl group} $\W$ is the group $N(\T)/\T$, where $N(\T)$ denotes the normalizer in $\G$.
        It is isomorphic to the product $\prod_{m=1}^l\mathfrak{S}_{r_m}$ of symmetric groups and acts on $\Fl(r_\bullet,V)_{\operatorname{ab}}$. Therefore, it also acts on the cohomology $H^*(\Fl(r_\bullet,V)_{\operatorname{ab}})$. Results of Ellingsrud-Str\o mme \cite{ES}, Martin \cite{Martin}, Kirwan \cite{Kirwan} relate the cohomology of a nonabelian quotient to the cohomology of the associated abelian quotient with its Weyl group action, see \cite[Section 3.1]{CFKS} for a more detailed description. In particular, there is a natural surjective ring homomorphism
        \begin{equation}\label{equation:cohomology_surjection}
        \varrho : H^*(\Fl(r_\bullet,V)_{\operatorname{ab}})^\W\twoheadrightarrow H^*(\Fl(r_\bullet,V))
        \end{equation}
        \item
        For each $1\leq m\leq l$, the $\T$-representation $E_{r_m}$ splits as a direct sum of $1$-dimensional representations, hence the bundle $(S_m^\vee)_\T$ splits into a direct sum of line bundles, $(S_m^\vee)_\T=\oplus_{j=1}^{r_m}M_j^{(m)}$.
        We denote the corresponding first Chern classes $H^{(m)}_j:=c_1(M_j^{(m)})\in H^*(\Fl(r_\bullet,V)_{\operatorname{ab}})$; these are represented by divisors on 
        $\Fl(r_\bullet,V)_{\operatorname{ab}}$ and we will also write $\mathcal{O}(H^{(m)}_i):=M^{(m)}_i$ for the line bundles. 
        The action of $\mathfrak{S}_{r_m}\subset\W$ on $H^*(\Fl(r_\bullet,V)_{\operatorname{ab}})$ permutes $H^{(m)}_1,\dots,H^{(m)}_{r_m}$.
        
        By abuse of notation, we use $H^{(m)}_1,\dots,H^{(m)}_{r_m}$ to denote the Chern roots of $S_m^\vee$ as well. This is justified, as the homomorphism $\varrho$ from \eqref{equation:cohomology_surjection} sends the Chern classes of $(S_m)_\T$ to the corresponding Chern classes of $S_m$, for each $1\leq m\leq l$. 
        
        For future reference, we describe explicitly the Lie algebra bundle  $\mathfrak{g}/\mathfrak{t}$ on $\Fl(r_\bullet,V)_{\operatorname{ab}}$ (with its decomposition $\mathfrak{g}/\mathfrak{t}=\oplus_{\alpha:\; {\mathrm{root\; of }}\;\G}L_\alpha$), which appears in the abelian/nonabelian correspondence, as explained in the Introduction. The Lie
        algebra $\mathfrak{g}=\oplus_{m=1}^l Lie(\mathrm{GL}_{r_m}(\C))$, viewed as the adjoint $\T$-representation, gives rise to the vector bundle
        $$\bigoplus_{m=1}^l\left((S_m^\vee)_\T\otimes(S_m)_\T\right)=\bigoplus_{m=1}^l\bigoplus_{1\leq i,j\leq r_m}\left(\mathcal{O}(H_i^{(m)})\otimes \mathcal{O}(-H_j^{(m)})\right),$$
        whose trivial subbundle of rank $r_1+\dots +r_l$ (coming from the triples $(m,i,j)$ with $i=j$) corresponds to the Cartan subalgebra $\mathfrak{t}$. Hence
        \begin{equation}\label{eqn:gmodt_explicit}
            \mathfrak{g}/\mathfrak{t}=\bigoplus_{m=1}^l\bigoplus_{1\leq i\neq j\leq r_m}\mathcal{O}\left( H_i^{(m)}-H_j^{(m)}\right).
        \end{equation}
        \item 
        Let 
        \begin{equation}\label{eqn:canonical_splitting}
        H^2(\Fl(r_\bullet,V)_{\operatorname{ab}},\Z)\cong H^2(B,\Z)\oplus\bigoplus_{m=1}^l\bigoplus_{i=1}^{r_m}\Z\cdot H^{(m)}_i
        \end{equation}
        be the canonical splitting \cite[Section 3.3]{Koto:mirror}.
        Its dual gives the embedding
        \[
        \NE(\Fl(r_\bullet,V)_{\operatorname{ab}}) \hookrightarrow \NE(B)\oplus\bigoplus_{m=1}^l\bigoplus_{i=1}^{r_m}\Z_{\geq0}\cdot d^{(m)}_i.
        \]
        We write $q^{(m)}_i$ ($1\leq m\leq l$, $1\leq i\leq r_m$) for the Novikov variable associated with the curve class $d^{(m)}_i$.
        We let $\W$ act on the Novikov ring by permuting $q^{(m)}_i$'s.
        \item 
        Taking the $\W$-invariant part of the splitting \eqref{eqn:canonical_splitting}, we obtain the splitting
        \[
        H^2(\Fl(r_\bullet,V),\Z)\cong H^2(B,\Z)\oplus\bigoplus_{m=1}^l \Z\cdot H^{(m)}
        \]
        where $H^{(m)}$ ($1\leq m\leq l$) denotes the sum $\sum_{i=1}^{r_m}H^{(m)}_i$, i.e., $H^{(m)}=c_1(S_m)$.
        Its dual gives the embedding
        \[
        \NE(\Fl(r_\bullet,V)) \hookrightarrow \NE(B)\oplus\bigoplus_{m=1}^l\Z_{\geq0}\cdot d^{(m)}.
        \]
        We write $q^{(m)}$ ($1\leq m\leq l$) for the Novikov variable associated with the curve class $d^{(m)}$.
    \end{enumerate}
\end{notation}

\subsection{Weyl-invariant Givental cone}\label{subsec:Weyl_invariant_cone_abelian_quotient}
We introduce a $\W$-invariant Givental cone for the abelian quotient $\Fl(r_\bullet,V)_{\operatorname{ab}}$.
This was first studied in \cite{CLS} in order to give a reformulation of the abelian/nonabelian correspondence using Givental's formalism.

We define $\cH_{\Fl(r_\bullet,V)_{\operatorname{ab}}}^\W$ and $\cL_{\Fl(r_\bullet,V)_{\operatorname{ab}}}^\W$ to be the $\W$-invariant part of $\cH_{\Fl(r_\bullet,V)_{\operatorname{ab}}}$ and $\cL_{\Fl(r_\bullet,V)_{\operatorname{ab}}}$, respectively.
More precisely, we set 
\[
\cH_{\Fl(r_\bullet,V)_{\operatorname{ab}}}^\W:=H^*(\Fl(r_\bullet,V)_{\operatorname{ab}})[z,z^{-1}][\![q_\bullet^\bullet,Q_B]\!]^\W.
\]
Here the $\W$-action is the combined one, from the action on $H^*(\Fl(r_\bullet,V)_{\operatorname{ab}})$ and the action on the Novikov variables $q_\bullet^\bullet$. Then, for any formal parameters $x=(x_1,x_2,\dots)$, the $\C[\![q_\bullet^\bullet,Q_B]\!]^\W[\![x]\!]$-valued points on $\cL_{\Fl(r_\bullet,V)_{\operatorname{ab}}}^\W$ are those points of $\cL_{\Fl(r_\bullet,V)_{\operatorname{ab}}}(\C[\![q_\bullet^\bullet,Q_B,x]\!])$ of the same form as \eqref{eqn:point_on_cone}, but with $\bt(z)\in H^*(\Fl(r_\bullet,V)_{\operatorname{ab}})[z][\![q_\bullet^\bullet,Q_B]\!]^\W[\![x]\!]$.
In other words, 
\[
\cL_{\Fl(r_\bullet,V)_{\operatorname{ab}}}^\W := \cL_{\Fl(r_\bullet,V)_{\operatorname{ab}}} \cap \cH_{\Fl(r_\bullet,V)_{\operatorname{ab}}}^\W.
\]

\subsubsection{Weyl-invariant equivariant Givental cone}
We consider the torus-equivariant counterpart of the $\W$-invariant Givental cone.

Let $\bbS$ be an algebraic torus acting on $\widetilde{V}$.
We assume that the actions of $\T$ and $\bbS$ on $\widetilde{V}$ are compatible: for any $t\in\T$, $s\in\bbS$ and $v\in\widetilde{V}$, we have $t\cdot(s\cdot v) = s\cdot(t\cdot v)$.
In this paper, we will deal with the case where $V=B\times\C^N$ is the trivial bundle of rank $N$, the rank of $\bbS$ is $N$, and $\bbS$ acts on $V$ diagonally. The action of $\bbS$ on $\widetilde{V}$ is then induced from the action on its summand $\Hom(\mathcal{O}_B^{\oplus r_l},V)$.

The \emph{$\W$-invariant $\bbS$-equivariant Givental cone} $\cL^\W_{\Fl(r_\bullet,V)_{\operatorname{ab}},\bbS}$ is defined as follows:
\begin{align*}
    \cH_{\Fl(r_\bullet,V)_{\operatorname{ab}},\bbS}^\W&:=H^*_\bbS(\Fl(r_\bullet,V)_{\operatorname{ab}})_\loc(\!(z^{-1})\!)[\![q_\bullet^\bullet,Q_B]\!]^\W, \\
    \cL_{\Fl(r_\bullet,V)_{\operatorname{ab}},\bbS}^\W &:= \cL_{\Fl(r_\bullet,V)_{\operatorname{ab}},\bbS} \cap \cH_{\Fl(r_\bullet,V)_{\operatorname{ab}},\bbS}^\W.
\end{align*}
As in the usual theory, any $H^*_\bbS(\pt)_\loc[\![q_\bullet^\bullet,Q_B]\!]^\W[\![x]\!]$-valued point on $\cL_{\Fl(r_\bullet,V)_{\operatorname{ab}},\bbS}^\W$ can be written as $M_{\Fl(r_\bullet,V)_{\operatorname{ab}}}^\bbS(\sigma(x))f$ where
\begin{align*}
    \sigma(x)&\in H^*_\bbS(\Fl(r_\bullet,V)_{\operatorname{ab}})_\loc[\![q^\bullet_\bullet,Q_B]\!]^\W[\![x]\!],  \\
    f&\in H^*_\bbS(\Fl(r_\bullet,V)_{\operatorname{ab}})_\loc[z][\![q^\bullet_\bullet,Q_B]\!]^\W[\![x]\!]
\end{align*}
with $\sigma(x)|_{q^\bullet_\bullet=Q_B=x=0}=0$ and $f|_{q^\bullet_\bullet=Q_B=x=0}=1$.

\subsection{Embeddings of flag bundles as zero loci}\label{subsection:embeddings}
In \cite{IK:quantum} and \cite{Koto:mirror}, mirror theorems for toric bundles (that is, for abelian quotients) were proved via embedding them as zero loci in larger {\em trivial} toric bundles. In this subsection, we explain that the construction can be made at the level of prequotients, and descended to both the nonabelian and the abelian quotients. 

Continuing to use the notation introduced in subsection \ref{subsection: flag_bundles_GIT}, suppose that $V$ is a subbundle in a vector bundle $U\to B$ of rank $N>n$ and let $\mathcal{Q}=U/V$ be the quotient vector bundle. Then we have the associated $\G$-equivariant vector bundle $\pi:\widetilde {U}\to B$, with $\widetilde{V}\subset \widetilde{U}$ a $\G$-subbundle and $\widetilde{U}/\widetilde{V}\cong\Hom(\cO_B^{\oplus r_l},\mathcal{Q})\cong\mathcal{Q}\otimes(\mathcal{O}_B\otimes \mathrm{St}_{r_l})$, where $\mathrm{St}_{r_l}$ is the standard representation of $GL_{r_l}(\C)$. (Note that $\G$ acts on $\widetilde{U}/\widetilde{V}$ via the projection $\G\to GL_{r_l}(\C)$.) 
The vector bundle 
$\pi^*\Hom(\cO_B^{\oplus r_l},U)$ on $\widetilde{U}$ has the tautological section $$\mathrm{taut}:\widetilde {U}\to \pi^*\Hom(\cO_B^{\oplus r_l},U).$$ 
Composing it with the vector bundle map $\pi^*\Hom(\cO_B^{\oplus r_l},U)\to \pi^*\Hom(\cO_B^{\oplus r_l},\mathcal{Q})$ gives a $\G$-equivariant section $\tilde{s}$ of $\pi^*\Hom(\cO_B^{\oplus r_l},\mathcal{Q})$, whose (transversal) zero locus in $\widetilde{U}$ is obviously $Z(\tilde{s})=\widetilde{V}$. The inclusion $\widetilde{V}\subset \widetilde{U}$ respects $\theta$-stability (for both $\G$ and $\T$). Passing to the GIT quotients, $\tilde{s}$ descends to sections 
$$s:\Fl(r_\bullet,U)\to p^*\mathcal{Q}\otimes S^\vee\;\;\;\; s_{\mathrm{ab}}:\Fl(r_\bullet,U)_{\mathrm{ab}}\to p_{\mathrm{ab}}^*\mathcal{Q}\otimes S_{\T}^\vee,$$
such that their transversal zero loci are 
$$Z(s)=\Fl(r_\bullet,V)\subset \Fl(r_\bullet,U),\;\;\;\; Z(s_{\mathrm{ab}})=\Fl(r_\bullet,V)_{\mathrm{ab}}\subset \Fl(r_\bullet,U)_{\mathrm{ab}}.$$
(Here $p:\Fl(r_\bullet,U)\to B$ and $p_{\mathrm{ab}}:\Fl(r_\bullet,U)_{\mathrm{ab}}\to B$ are the canonical projections, $S=S_l$ is the last tautological subbundle on $\Fl(r_\bullet,U)$, and $S_\T$ is its ``abelianization".)

It is clear that the tautological flag of subbundles on $\Fl(r_\bullet,U)$ restricts to the tautological flag on $\Fl(r_\bullet,V)$, so we will use the notation $S_1\subset S_2\subset\cdots\subset S_l=:S$ in either case. Similarly, each $(S_m)_\T$ on $\Fl(r_\bullet,U)_{\mathrm{ab}}$ restricts to the corresponding bundle $(S_m)_\T$ on $\Fl(r_\bullet,V)_{\mathrm{ab}}$. The cohomology classes $H_j^{(m)}$ therefore correspond under pull-back, as do the Chern classes of the $S_m$'s. Finally, push-forward for both inclusions is an isomorphism on $H_2(-,\Z)$, so the Novikov rings are canonically isomorphic.

\section{Mirror theorem}\label{sec:main_theorem}

Let $B$ be a smooth projective variety, let $V\to B$ be a vector bundle of rank $n$ whose dual is generated by global sections, and let $r_\bullet=\{r_1,\dots,r_l\}$ be a subset of $\{1,2,\dots,n-1\}$ with $r_1<\dots<r_l$.
We use Notation \ref{notation:flag_bundle}.

The purpose of this section is to state and prove the mirror theorem for $\Fl(r_\bullet,V)$.
Before proceeding, we recall a well-known fact about equivariant cohomology:
for any $\G$-space $X$, the Weyl group $\W$ acts on $H^*_\T(X)$, and we have an isomorphism $H_\G^*(X)\cong H_\T^*(X)^\W$. In particular, we can speak of $H^*_\T(\widetilde{V})^\W$, though in this case the action is easy to describe directly. Namely, for each $1\leq m\leq l$, let $\lambda_1^{(m)},\dots, \lambda_{r_m}^{(m)}$ denote the equivariant parameters for the factor of $\T$ given by the maximal torus in $GL_{r_m}(\C)$, so that $H^*_\T(\mathrm{pt})=\C[\lambda_{\bullet}^\bullet]$. Then $H^*_\T(\widetilde{V})\cong H^*(B)[\lambda_{\bullet}^\bullet]$ and $\W=\prod_{m=1}^l\mathfrak{S}_{r_m}$ acts by permuting the 
$\lambda_{\bullet}^\bullet$'s in the obvious way.

\begin{theorem}[mirror theorem]\label{thm:mirror_Fl}
    Assume that the dual of $V$ is globally generated.
    Let $z\cdot I^{\lambda_\bullet^\bullet}_{\widetilde{V}}(x,z) \in H^*_\T(\widetilde{V})^\W[z,z^{-1}][\![Q_B,x]\!]$ be a $\W$-invariant point on the Givental cone $\cL_{\widetilde{V}}$ (in particular, $I^{\lambda_\bullet^\bullet}_{\widetilde{V}}$ depends on $\lambda_\bullet^\bullet$ polynomially). 
    Define the $H^*(\Fl(r_\bullet,V))$-valued function $I_{\Fl(r_\bullet,V)}(x,z)$ to be
    \begin{multline*}
    \sum_{ \substack{k^{(m)}\geq0; \\1\leq m\leq l} } \left[ \prod_{m=1}^l (q^{(m)})^{k^{(m)}}e^{\left(k^{(m)}+\frac{H^{(m)}}{z}\right)t^{(m)}} \right] \cdot \sum_{ \substack{k_i^{(m)}\geq0; 1\leq m\leq l, \\ 1\leq i\leq r_m,\sum_ik_i^{(m)}=k^{(m)}} } \frac{ I^{H_\bullet^\bullet+k_\bullet^\bullet z}_{\widetilde{V}}(x,z) }{\prod_{i=1}^{r_l} \prod_{c=1}^{k_i^{(l)}} \prod_{ \substack{\delta\colon\mathrm{Chern\ roots} \\ \mathrm{of\ } V} } (H_i^{(l)} + \delta + cz)}	\\
    \cdot \prod_{m=1}^{l-1}\prod_{i=1}^{r_m}\prod_{j=1}^{r_{m+1}}\frac{\prod_{c=-\infty}^{0}(H_i^{(m)}-H_j^{(m+1)}+cz)}{\prod_{c=-\infty}^{k_i^{(m)}-k_j^{(m+1)}}(H_i^{(m)}-H_j^{(m+1)}+cz)}
    \cdot \prod_{m=1}^l\prod_{i\neq j}\frac{\prod_{c=-\infty}^{k_i^{(m)}-k_j^{(m)}}(H_i^{(m)}-H_j^{(m)}+cz)}{\prod_{c=-\infty}^{0}(H_i^{(m)}-H_j^{(m)}+cz)}
    \end{multline*}
    Then, $z\cdot I_{\Fl(r_\bullet,V)}(x,z)$ lies in the Givental cone $\cL_{\Fl(r_\bullet,V)}$.
\end{theorem}

\begin{proof}
    Since $V^\vee$ is globally generated, we have a short exact sequence of vector bundles on $B$
    \begin{equation}\label{eqn:exact_sqn}
    0\to V\to \cO_B^{\oplus N}\to \mathcal{Q}\to 0
    \end{equation}
    for some $N\geq\operatorname{rank}V$ and some vector bundle $\mathcal{Q}\to B$.
    Therefore, we can apply the construction of subsection \ref{subsection:embeddings}, taking $U=\cO_B^{\oplus N}$, so that $\Fl(r_\bullet, U)$ is the product $B\times\Fl(r_\bullet,N)$.
    We obtain embeddings as zero loci
    \[\iota\colon\Fl(r_\bullet,V)= Z(s)\hookrightarrow B\times\Fl(r_\bullet,N),\;\;\;\;\iota_{\mathrm{ab}}\colon\Fl(r_\bullet,V)_{\mathrm{ab}}= Z(s_{\mathrm{ab}})\hookrightarrow B\times\Fl(r_\bullet,N)_{\mathrm{ab}}.\]
    Here $s$ is a section of the vector bundle $\mathcal{Q}\boxtimes S^\vee$ on $B\times\Fl(r_\bullet,N)$, with
 $S\to\Fl(r_\bullet,N)$ the tautological vector bundle of rank $r_l$, and $s_{\mathrm{ab}}$ is the corresponding section of $\mathcal{Q}\boxtimes (S_\T)^\vee$ on $B\times\Fl(r_\bullet,N)_{\mathrm{ab}}$. (In this case, the section $s$ corresponds to
 the homomorphism of vector bundles on $B\times\Fl(r_\bullet,N)$ given by the composition 
    $$p_2^*S\to \cO_{B\times\Fl(r_\bullet,N)}^{\oplus N}\to p_1^*\mathcal{Q},$$
    with $p_1$ and $p_2$ the projections on the two factors.)

    Now consider the $H^*(B\times\Fl(r_\bullet,N))$-valued function
    \begin{multline}\label{eqn:twisted_F}
        F(\mu) = \sum_{ \substack{k^{(m)}\geq0; \\1\leq m\leq l} } \left[ \prod_{m=1}^l (q^{(m)})^{k^{(m)}}e^{\left(k^{(m)}+\frac{H^{(m)}}{z}\right)t^{(m)}} \right] \cdot \sum_{ \substack{k_i^{(m)}\geq0; 1\leq m\leq l, \\ 1\leq i\leq r_m,\sum_ik_i^{(m)}=k^{(m)}} } I^{\mu+H_\bullet^\bullet+k_\bullet^\bullet z}_{\widetilde{V}}(x,z) \\ 
        \cdot \prod_{i=1}^{r_l}\prod_{c=1}^{k_i^{(l)}}\frac{ \prod_{ \substack{\epsilon\colon\mathrm{Chern\ roots} \\ \mathrm{of\ } \mathcal{Q}} } (\mu+H_i^{(l)}+\epsilon+cz) }{(H_i^{(l)}+cz)^N} 
        \cdot \prod_{m=1}^{l-1}\prod_{i=1}^{r_m}\prod_{j=1}^{r_{m+1}}\frac{\prod_{c=-\infty}^{0}(H_i^{(m)}-H_j^{(m+1)}+cz)}{\prod_{c=-\infty}^{k_i^{(m)}-k_j^{(m+1)}}(H_i^{(m)}-H_j^{(m+1)}+cz)} \\
        \cdot \prod_{m=1}^l\prod_{i\neq j}\frac{\prod_{c=-\infty}^{k_i^{(m)}-k_j^{(m)}}(H_i^{(m)}-H_j^{(m)}+cz)}{\prod_{c=-\infty}^{0}(H_i^{(m)}-H_j^{(m)}+cz)}
    \end{multline}

\begin{lemma}\label{lem:twisted_function}
    The function $z\cdot F(\mu)$ lies in the twisted Givental cone $\cL_{B\times\Fl(r_\bullet,N),(\mathcal{Q}\boxtimes S^\vee,e_\mu)}$.
\end{lemma}

We will prove this Lemma in Section \ref{sec:proof_of_Lemma}. Granting it for now, the proof of Theorem \ref{thm:mirror_Fl} can be readily completed. Indeed, the non-equivariant limit $\lim_{\mu\to0} \iota^*F(\mu)$ clearly exists and equals $\iota^*F(0)$. On the other hand, as $V$ and $\mathcal{Q}$ are related by the exact sequence \eqref{eqn:exact_sqn}, the splitting principle (\cite[$\S21$]{BT:differential}, \cite[Remark 3.2.3]{Fulton}) shows that $\iota^*F(0)=I_{\Fl(r_\bullet,V)}(x,z)$.  By Lemma \ref{lem:I-function_of_subvarieties}, we conclude that $z\cdot I_{\Fl(r_\bullet,V)}(x,z)$ lies in $\cL_{\Fl(r_\bullet,V)}$, as needed\footnote{Recall that $\iota_*:H_2(\Fl(r_\bullet,V),\Z)\to H_2(B\times\Fl(r_\bullet,N),\Z)$ is an isomorphism, so the Novikov ring is canonically identified with $\C[\![q^\bullet,Q_B]\!]$ for both $B\times\Fl(r_\bullet,N)$ and its subvariety $\Fl(r_\bullet,V)$.}.
     
\end{proof}

\begin{example}\label{ex:IK}
    If we set $l=1$ and $r_\bullet=\{r\}$, then Theorem \ref{thm:mirror_Fl} specializes to Theorem \ref{thm:mirror_Gr}.
    If we further set $r=1$, so that we are looking at the projective bundle $\mathbb{P}(V)$, then the $I$-function matches the one given by \cite[Theorem 3.3]{IK:quantum}.
\end{example}

\begin{example}\label{ex:Oh}
    Theorem \ref{thm:mirror_Fl} recovers Oh's mirror theorem, as we now explain.
    Let $V=\bigoplus_{j=1}^nL_j$ be a split vector bundle with $L_j^\vee$ ($1\leq j\leq n$) being globally generated.
    Let $J_B(\tau)=\sum_{d\in\NE(B)} J_d(\tau) Q_B^d$ be the $J$-function of $B$, and consider the $H^*(B)$-valued function
    \[
    I_{\widetilde{V}}^{\lambda_\bullet^\bullet} = \sum_{d\in\NE(B)} \left( \prod_{j=1}^n \prod_{i=1}^{r_l} \prod_{c=0}^{-c_1(L_j)\cdot d-1} (\lambda_i^{(l)} + c_1(L_j) - cz) \right) J_d(\tau) Q_B^d.
    \]
    It follows from the quantum Lefschetz theorem \cite{CG:quantum} that $z\cdot I_{\widetilde{V}}^{\lambda_\bullet^\bullet}$ lies in $\cL_{\widetilde{V},\T}$.
    Then, the $I$-function of $\Fl(r_\bullet,V)$ associated to $I_{\widetilde{V}}^{\lambda_\bullet^\bullet}$ coincides with Oh's $I$-function \cite{Oh:split_flag_bundle}.
\end{example}

It remains to prove Lemma \ref{lem:twisted_function}. In the argument we give in Section \ref{sec:proof_of_Lemma}, the \emph{full} abelian/nonabelian correspondence for trivial flag bundles plays an important role. The next section is dedicated to establishing it (in Theorem \ref{thm:abelian/nonabelian_correspondence}).

\section{Abelian/Nonabelian correspondence for trivial flag bundles}\label{sec:ab/nonab}

The abelian/nonabelian correspondence \cite{BCFK:ab/nonab,CFKS} (see also \cite{Webb,CLS} for later developments) is a conjectural relationship between the genus zero Gromow-Witten theory of a nonabelian GIT quotient and the genus zero Gromov-Witten theory of its associated abelian quotient.
Instead of discussing it in the general situation, we only establish a version of it for $B\times\Fl(r_\bullet,N)$ (Theorem \ref{thm:abelian/nonabelian_correspondence}) as that will be sufficient for the proof of Lemma \ref{lem:twisted_function}.

\subsection{Statement}\label{subsec:statement_of_ab/nonab}
Throughout this section, we fix a natural number $N$ and a subset $r_\bullet=\{r_1,\dots,r_l\}$ of $\{1,2,\dots,N-1\}$, with $r_1<\cdots<r_l$.

\begin{notation}\label{notation:S-action}
    \phantom{A}
    \begin{enumerate}
        \item 
        We write $\FlG$ for $\Fl(r_\bullet,N)$, and write $\FlT$ for $\Fl(r_\bullet,N)_{\operatorname{ab}}$.
        \item 
        Fix a $\C$-basis $\{\phi_i\}_{i\in I}$ of $H^*(B,\C)$ and write $\tau=\{\tau^a\}_{a\in I}$ for the dual coordinates.
        \item 
        Let $\bbS$ be an algebraic torus of rank $N$ acting on $\mathcal{O}_B^{\oplus N}$ diagonally.
        This induces $\bbS$-actions on $B\times\FlG$ and on $B\times\FlT$ (where the $\bbS$-action on $B$ is trivial) and the projections $B\times\FlG\to B$ and $B\times\FlT\to B$ are $\bbS$-equivariant. (The $\bbS$-action on $\FlG$ is the standard action of the maximal torus in $GL_N(\C)$; it has isolated fixed points and isolated $1$-dimensional orbits.)
        \item 
        We write $\nu_1,\dots,\nu_N$ for the $\bbS$-equivariant parameters.
    \end{enumerate}
\end{notation}

The $\bbS$-equivariant Givental space of $B\times\FlG$ and the $\W$-invariant part of the $\bbS$-equivariant Givental space of $B\times\FlT$ are 
\begin{align*}
\cH_{B\times\FlT,\bbS}^\W&=H^*_\bbS(B\times\FlT)_\loc(\!(z^{-1})\!)[\![q_\bullet^\bullet]\!]^\W[\![Q_B]\!], \\
\cH_{B\times\FlG,\bbS}&=H^*_\bbS(B\times\FlG)_\loc(\!(z^{-1})\!)[\![q^\bullet,Q_B]\!].
\end{align*}

Recall from \eqref{eqn:gmodt_explicit} that on the abelian quotient $B\times\FlT$ we have the Lie algebra bundle $\mathfrak{g}/\mathfrak{t}=\bigoplus_{m=1}^l\bigoplus_{1\leq i\neq j\leq r_m}\mathcal{O}\left( H_i^{(m)}-H_j^{(m)}\right)$. Following \cite{BCFK:ab/nonab},
for any effective curve class whose component in the $\FlT$-direction is $\sum_{m,i}k_i^{(m)}d_i^{(m)}$ ($k_i^{(m)}\in\Z_{\geq 0}$), the hypergeometric modification factor associated to $\mathfrak{g}/\mathfrak{t}$ is 
\begin{equation} \label{eqn:hypergeometric}
    \Hyp_{k^\bullet_\bullet}^{\mathfrak{g}/\mathfrak{t}}(H^\bullet_\bullet,z) := \prod_{m=1}^l\prod_{i\neq j}\frac{\prod_{c=-\infty}^{k_i^{(m)}-k_j^{(m)}}(H_i^{(m)}-H_j^{(m)}+cz)}{\prod_{c=-\infty}^{0}(H_i^{(m)}-H_j^{(m)}+cz)}.
\end{equation}
We now define the map 
$$
(\cdot)_\gt\colon\cH_{B\times\FlT,\bbS}^\W\to\cH_{B\times\FlG,\bbS}
$$
as follows: for any point
\[
I = \sum_{ \substack{k_i^{(m)}\geq0; \\ 1\leq m\leq l,  1\leq i\leq r_m} } \left(\prod_{m=1}^l\prod_{i=1}^{r_m} (q_i^{(m)})^{k_i^{(m)}}\right)\cdot I_{k^\bullet_\bullet}(Q_B,x,z)\in\cH_{B\times\FlT,\bbS}^\W
\]
with $I_{k_\bullet^\bullet}(Q_B,x,z)\in H^*_\bbS(B\times\FlT)_\loc(\!(z^{-1})\!)[\![Q_B,x]\!]$, we set
\begin{align}\label{eqn:g/t}
    I_\gt &= \sum_{k^\bullet_\bullet} \left( \prod_{m=1}^l (q^{(m)})^{\sum_{i=1}^{r_m}k_i^{(m)}}\right)\cdot \left(I_{k_\bullet^\bullet}(Q_B,x,z) \cdot \Hyp_{k^\bullet_\bullet}^{\mathfrak{g}/\mathfrak{t}}(H^\bullet_\bullet,z)\right).  
\end{align}
In words, we first multiply each term in the $q_\bullet^\bullet$-exapnsion of $I$ by the appropriate $\mathfrak{g}/\mathfrak{t}$-modification factor, and then we specialize the Novikov variables by setting $q_j^{(m)}=q^{(m)}$, for $j=1,\dots ,r_m$, $m=1, \dots ,l$. 
It is not hard to check that the formula \eqref{eqn:g/t} gives a well-defined element of $\cH_{B\times\FlG,\bbS}(\C[\![q^\bullet,Q_B,x]\!])$ (for example, see \cite[Lemma 4.1]{CLS}, where $I_\gt$ is called the Givental-Martin modification of $I$).
The non-equivariant version 
\[
(\cdot)_\gt\colon\cH_{B\times\FlT}^\W\to\cH_{B\times\FlG}
\] 
is defined by the same formula.

An $H^*_\bbS(\pt)_\loc[\![Q_B,q_\bullet^\bullet,t_\bullet^\bullet,x]\!]$-valued point $I(t_\bullet^\bullet,x,z)$ on $\cH_{B\times\FlT,\bbS}$ is said to satisfy \emph{the Divisor Equation} \cite{CLS} if, for any $1\leq m\leq l$ and $1\leq i\leq r_m$, the following equality holds:
\begin{equation}\label{eqn:Divisor_Equation}
    z\parfrac{I}{t_i^{(m)}}(t_\bullet^\bullet,x,z) = zq_i^{(m)}\parfrac{I}{q_i^{(m)}}(t_\bullet^\bullet,x,z) + H_i^{(m)}\cdot I(t_\bullet^\bullet,x,z).
\end{equation}

The aim of this section is to prove the following version of the abelian/nonabelian correspondence for $B\times\Fl$.

\begin{theorem}\label{thm:abelian/nonabelian_correspondence}
    Let $z\cdot I(t^\bullet_\bullet,x,z) \in \cL_{B\times\FlT,\bbS}$ be a point satisfying the Divisor Equation whose specialization $I(t^\bullet,x,z)$ is $\W$-invariant.
    Then, $z \cdot (I(t^\bullet,x,z))_\gt$ lies in $\cL_{B\times\Fl,\bbS}$.
\end{theorem}

The above formulation (for general GIT quotients) is due to Coates-Lutz-Shafi, see \cite[Conjecture 1.8]{CLS}, where it is called a reformulation of \cite[Conjecture 3.7.1]{CFKS}

We will prove Theorem \ref{thm:abelian/nonabelian_correspondence} in \S~\ref{subsection:ab/nonab}.
The argument relies on a characterization theorem for $\cL_{B\times\Fl,\bbS}$ (Theorem~\ref{thm:characterization}), which is used to reduce to the case when $B$ is a point. In that case, the statement follows the abelian/nonabelian correspondence for partial flag varieties, which was proved in \cite[Theorem 4.1.1]{CFKS}.

In the proof of Lemma \ref{lem:twisted_function}, we need the abelian/nonabelian correspondence for twisted theories \cite[Conjecture 6.1.1]{CFKS}, \cite[Conjecture 1.9]{CLS}, which is in fact equivalent to the one for untwisted theory. 
More precisely, by the same argument as in the proofs of \cite[Theorem 6.1.2]{CFKS} and \cite[Proposition 1.7]{CLS}, Theorem~\ref{thm:abelian/nonabelian_correspondence} implies the following.

\begin{corollary}\label{cor:twisted_abelian_nonabelian_correspondence}
    Let $E$ be a $\G$-bundle over $\widetilde{V}$, and let $E_\G$ (resp. $E_\T$) be the induced vector bundle over $B\times\FlG$ (resp. $B\times\FlT$).
    Let $z\cdot I(t^\bullet_\bullet,x,z) \in \cL_{B\times\FlT,(E_\T,e_\mu)}$ be a point satisfying the Divisor Equation whose specialization $I(t^\bullet,x,z)$ is $\W$-invariant.
    Then, $z \cdot (I(t^\bullet,x,z))_\gt$ lies in $\cL_{B\times\Fl,(E_\G,e_\mu)}$.
\end{corollary}

\subsection{Characterization of points on $\cL_{B\times\FlG,\bbS}$}
We introduce the Givental and Brown style characterization of points on $\cL_{B\times\FlG,\bbS}$.
See \cite{Brown,CCIT:mirror,JTY,FL,Koto:mirror} for similar characterizations (for other varieties/stacks).

We prepare notation.
\begin{itemize}
    \item 
    Let $F$ denote the set of all $\bbS$-fixed points of $\FlG$.
    \item 
    For $\alpha\in F$, we define $\adj(\alpha)$ to be the subset of $F\setminus\{\alpha\}$ consisting of all $\bbS$-fixed points that can be joined with $\alpha$ by a $1$-dimensional orbit.
    Such an orbit is unique for each $\beta\in\adj(\alpha)$, and we write $d_{\alpha,\beta}$ for the curve class of its closure.
    \item 
    For any $\beta\in\adj(\alpha)$, let $\rho_{\alpha,\beta}\in H^2_\bbS(\pt,\Z)$ denote the $\bbS$-equivariant first Chern class of the conormal bundle of $\alpha\in\FlG$ in the closure of the one-dimensional orbit joining $\alpha$ and $\beta$.
    \item 
    For each $\alpha\in F$, write $\iota_\alpha\colon B\hookrightarrow B\times\FlG$ for the embedding that intersects with each fiber of $B\times\FlG\to B$ at $\alpha\in\FlG$.
\end{itemize}

The following gives the characterization of points on $\cL_{B\times\FlG,\bbS}$.

\begin{theorem}\label{thm:characterization}
    Let 
    \[
    I(x,z)\in H^*_\bbS(B\times\FlG)(\!(z^{-1})\!)[\![q^\bullet,Q_B,x]\!]
    \]
    be a function such that $I|_{q^\bullet=Q_B=x=0}=1$.
    Then, the function $z\cdot I(x,z)$ is a $H^*_\bbS(\pt)_\loc[\![q^\bullet,Q_B,x]\!]$-valued point of $\cL_{B\times\FlG,\bbS}$ if and only if $I(x,z)$ satisfies the following three properties:
    \begin{itemize}
        \item[$\mathbf{(C1)}$]
        For any $\alpha\in F$, the restriction $\iota_\alpha^*I$ belongs to
        \[
         H^*_{\bbS}(B)_\loc (z) [\![q^\bullet,Q_B,x]\!],
        \]
        and it is a regular function in $z$ except possibly for poles at
        \[
        \{0,\infty\} \cup \left\{ \frac{\rho_{\alpha,\beta}}{a} \colon \beta\in\adj(\alpha), a\in\N \right\}.
        \]
        
        \item[$\mathbf{(C2)}$]
        The restrictions of $I(x,z)$ satisfy the following recursion relations: for any $\alpha\in F$, $\beta\in\adj(\alpha)$ and $a\in\N$, the following equality holds.
        \[
        \Prin_{z=\frac{\rho_{\alpha,\beta}}{a}} \iota_\alpha^*I = \frac{q^{a\cdot d_{\alpha,\beta}}}{z-\frac{\rho_{\alpha,\beta}}{a}} \cdot \frac{e_\bbS(N_\alpha)}{e_\bbS(N^\vir_{\alpha,\beta,a})} \cdot \iota_\beta^*I\left( z=\frac{\rho_{\alpha,\beta}}{a} \right)
        \]
        where $N_\alpha$ is the normal bundle to the $\bbS$-fixed point $\alpha$ in $\FlG$, and $N^\vir_{\alpha,\beta,a}$ is the virtual normal bundle to $\ovcM_{\alpha,\beta,a}$, the moduli of all multiple one-dimensional orbits joining $B\times\{\alpha\}$ and $B\times\{\beta\}$ \cite[Definition 3.10]{Koto:mirror}, in $\ovcM_{0,2}(B\times\FlG,a\cdot d_{\alpha,\beta})$.
        
        \item[$\mathbf{(C3)}$]
        For any $\alpha\in F$, the Laurent expansion of $\iota_\alpha^*I(x,z)$ at $z=0$ is a $\C(\nu)[\![q^\bullet,Q_B,x,\{(\rho_{\alpha,\beta})^{-1}\}_{\beta\in\adj(\alpha)}]\!]$-valued point of $\cL_B$.
    \end{itemize}
\end{theorem}

\begin{remark}\label{rmk:characterization}\phantom{A}
    \begin{itemize}
        \item[$(1)$]
        The term $e_\bbS(N_\alpha)/e_\bbS(N^\vir_{\alpha,\beta,a}) \in H^*_\bbS(\pt)_\loc$ does not depend on the base space $B$.
        Though it is not necessary for this paper, it can be explicitly computed using the explicit formula for the small $J$-function of partial flag varieties from \cite{BCFK:ab/nonab}.
        \item[$(2)$] 
        One can easily generalize this characterization to the case of nontrivial, but still {\em split} partial flag bundles over $B$.
        In fact, the proof given below will work for any split partial flag bundle. 
        \item[$(3)$]
        Motivated by \cite{Givental:mirror,CFK:wall}, Oh gave a characterization \cite[Theorem 1.3]{Oh:split_flag_bundle} of points (satisfying some additional condition) on the equivariant Givental cone of split partial flag bundles.
        Theorem \ref{thm:characterization} and its version for general split bundles may be viewed as a slight extension of Oh's characterization. 
    \end{itemize}
\end{remark}

\begin{proof}[Proof of Theorem \ref{thm:characterization}]
    We first check that any point $I(x,z)$ on $z^{-1}\cL_{B\times\FlG,\bbS}$ satisfies the three properties.
    One can see that $\ovcM_{\alpha,\beta,a}$ is isomorphic to the $a$-th root stack $B\times[\C^\times_s/\C^\times_t]$ where $\C^\times_t$ acts on $\C^\times_s$ as $t\cdot s=t^as$; see e.g. \cite[Definition 3.10]{Koto:mirror}.
    In this situation, the virtual localization theorem \cite{GP} shows that $I(x,z)$ satisfies $\mathbf{(C1)}$ and $\mathbf{(C2)}$; see \cite{Brown,Oh:split_flag_bundle}.
    Finally, $\mathbf{(C3)}$ follows from \cite[Proposition 4.2]{Iritani:blowups}.

    Conversely, we assume that $I(x,z)$ satisfies $\mathbf{(C1)}$, $\mathbf{(C2)}$ and $\mathbf{(C3)}$.
    It follows from the argument in \cite[Section 4.5]{Koto:mirror} that the function $I(x,z)$ is uniquely determined from the non-negative part $z^{-1}\Prin_{z=\infty}z\cdot I$.
    On the other hand, as we checked above, the point $I'$ on $z^{-1}\cL_{B\times\FlG,\bbS}$ whose non-negative part coincides with that of $I$ satisfies the three properties.
    This shows that $I=I'$ is a point on $z^{-1}\cL_{B\times\FlG,\bbS}$.
\end{proof}

\subsection{$I$-function of trivial toric bundles}\label{subsec:Brown}
We review next a (relatively small) $I$-function of $B\times\FlT$.
Brown \cite{Brown} gives a formula for an $I$-function (i.e., a point on the Givental cone) of a \emph{split} toric bundle.
In particular, it applies to trivial toric bundles, and in that case it takes a simple form. 
Namely, define $I^\bbS_{B\times\FlT}(\tau,t_\bullet^\bullet,z)$ to be the product 
\begin{equation}\label{eqn:brown_function}
I^\bbS_{B\times\FlT}(\tau,t_\bullet^\bullet,z):=J_B(\tau,z)\cdot I_{\FlT}^\bbS(t_\bullet^\bullet,z),
\end{equation}
where $J_B(\tau,z)$ is the $J$-function \eqref{eqn:J-function} of $B$ and $I_{\FlT}^\bbS(t_\bullet^\bullet,z)$ is Givental's $\bbS$-equivariant small $I$-function of the toric variety $\FlT$ (see, e.g., \cite{Givental:mirror}). 
Explicitly,
\begin{align*}
    I_{\FlT}^\bbS &= \sum_{ \substack{k_i^{(m)}\geq0; \\1\leq m\leq l, 1\leq i\leq r_m} } \left[ \prod_{m=1}^l\prod_{i=1}^{r_m}(q_i^{(m)})^{k_i^{(m)}}e^{\left(k_i^{(m)}+\frac{H_i^{(m)}}{z}\right)t_i^{(m)}} \right] \cdot I_{k_\bullet^\bullet}^\bbS(H_\bullet^\bullet,z)  \\
    I_{k_\bullet^\bullet}^\bbS(H_\bullet^\bullet,z) &= \prod_{m=1}^{l}\prod_{i=1}^{r_m}\prod_{j=1}^{r_{m+1}}\frac{\prod_{c=-\infty}^{0}(H_i^{(m)}-H_j^{(m+1)}+cz)}{\prod_{c=-\infty}^{k_i^{(m)}-k_j^{(m+1)}}(H_i^{(m)}-H_j^{(m+1)}+cz)}
\end{align*}
Here we set $r_{l+1}:=N$, $H_j^{(l+1)}:=\nu_j$ and $k_j^{(l+1)}:=0$.
Then Brown's result says that $z\cdot I_{B\times\FlT}^\bbS$ lies in $\cL_{B\times\FlT,\bbS}$.
Further, the non-equivariant limit 
\begin{equation}
I_{B\times\FlT} := \lim_{\nu\to0} I^\bbS_{B\times\FlT}
\end{equation}
gives a point on $z^{-1}\cL_{B\times\FlT}$ as well. Note that the non-equivariant $I_{B\times\FlT}$ obviously exists: it is obtained simply by setting $H_j^{(l+1)}=0$ for all $1\leq j\leq r_{l+1}$.

By inspection, the specialization $[z\cdot I_{B\times\FlT}^\bbS]_{t_i^{(m)}\to t^{(m)}}$ is a $H^*_\bbS(\pt)[\![q_\bullet^\bullet]\!]^\W[\![Q_B,\tau,t^\bullet]\!]$-valued point of $\cL_{B\times\FlT,\bbS}^\W$.

\subsection{Abelian/Nonabelian correspondence}\label{subsection:ab/nonab}

This section is devoted to proving Theorem~\ref{thm:abelian/nonabelian_correspondence}.

We introduce notation used in the rest of this section. 
Fix a homogeneous $H^*_\bbS(\pt)$-basis $\{P_j(H_\bullet^\bullet)\}_{j\in J}$ of $H^*_\bbS(\FlT)^\W$ such that $P_0=1$ where $P_j$ ($j\in J$) is a $\W$-invariant polynomial.
Let $\{\sigma^j\}_{j\in J}$ denote the associated coordinates on $H^*_\bbS(\FlT)^\W$, and set $\sigma=\sum_{j\in J}\sigma^jP_j(H_\bullet^\bullet)$.
For the root system $\Phi$, choose a subset $\Phi_+$ of positive roots, and write $\omega:=\prod_{\alpha\in\Phi_+}c_1(L_\alpha)$.
Recall that there is the following short exact sequence
\[
0 \to \ker(\omega\cup-) \to H^*_\bbS(B\times\FlT)^\W \to H^*_\bbS(B\times\FlG) \to 0,
\]
and it induces the identification $H^*_\bbS(B\times\FlG) \cong H^*_\bbS(B\times\FlT)^{-W}$ that sends $\beta\in H^*_\bbS(B\times\FlG)$ to $\widetilde{\beta}\cup\omega$ where $\widetilde{\beta}\in H^*_\bbS(B\times\FlT)^\W$ is any lift of $\beta$; see \cite{Martin}.

\begin{notation}\label{cup omega}
    
For $\beta\in H^*_\bbS(B\times\FlG)$ and $\gamma \in H^*_\bbS(B\times\FlT)^{-W}$, the notation $\beta \cup \omega = \gamma$ means that $\beta$ corresponds to $\gamma$ via the identification $H^*_\bbS(B\times\FlG) \cong H^*_\bbS(B\times\FlT)^{-W}$.
We further abuse this notation by writing $\gamma\cup\omega^{-1} = \beta$ when convenient.

\end{notation}

We first consider the case where $B$ is a point.
The abelian/nonabelian correspondence in terms of formal $\bbS$-equivariant Frobenius manifolds \cite[Conjecture 3.7.1]{CFKS} was established for partial flag varieties \cite[Theorem 4.1.1]{CFKS}.
In particular, it yields the following.

\begin{proposition}\label{prop:correspondence_J-function}
    There exists $\epsilon(\sigma)\in H^*_\bbS(\Fl)[\![q^\bullet,\sigma]\!]$ such that the function $M_{\FlT,\bbS}(\sigma)\omega|_{q^{(m)}_i\to (-1)^{r_m-1} q^{(m)}}$ equals $M_{\FlG,\bbS}(\epsilon(\sigma))1$ via the canonical identification $H^*_\bbS(\FlT)^{-\W}\cong H^*_\bbS(\Fl)$. 
\end{proposition}

\begin{proof}
    Fix a lift $\varphi_0\colon H^*_\bbS(\FlG)\to H^*_\bbS(\FlT)^\W$ of the canonical morphism $H^*_\bbS(\FlT)^\W\to H^*_\bbS(\FlG)$.
    An $\bbS$-equivariant analogue of \cite[Lemma 5.3.1]{CFKS} shows that the equality 
    \begin{equation}\label{old CFKS}
      M_{\FlT,\bbS}(\varphi(\epsilon))\omega = J_{\FlG,\bbS}(\epsilon) \cup \omega 
    \end{equation}
    holds after the specialization $q^{(m)}_i\to (-1)^{r_m-1} q^{(m)}$, where $\epsilon$ denotes an $H^*_{\bbS}(\pt)$-coordinate system on $H^*_\bbS(\Fl)$, and $\varphi(\epsilon)\in H^*_\bbS(\FlT)^\W[\![q^\bullet,\epsilon]\!]$ with $\varphi|_{q^\bullet=0} = \varphi_0$. (In the right-hand side, the notation \ref{cup omega} was used.)

    Choose a $H^*_{\bbS}(\pt)[\![q^\bullet,\epsilon]\!]$-basis $(a_1,a_2,\dots)$ of the kernel of the quantum product 
    \[
    \omega\star^{\FlT}-\colon H^*_\bbS(\FlT)^\W[\![q^\bullet,\epsilon]\!] \to H^*_\bbS(\FlT)^{-\W}[\![q^\bullet,\epsilon]\!], 
    \]
    and denote the associated coordinate system by $\eta=(\eta^1,\eta^2,\dots)$.
    We claim that $M_{\FlT,\bbS}(\varphi(\epsilon)+\eta)\omega=M_{\FlT,\bbS}(\varphi(\epsilon))\omega$. Indeed, this follows from
    \[
    z\partial_{\eta^i}M_{\FlT,\bbS}(\varphi(\epsilon)+\eta)\omega = M_{\FlT,\bbS}(\varphi(\epsilon)) (a_i\star^{\FlT}\omega)=0,
    \]
    which holds for every $i$ by \eqref{fundamental solution}.
    Combining the claim with equation \eqref{old CFKS}, we conclude that
    \[
    [M_{\FlT,\bbS}(\varphi(\epsilon)+\eta)\omega]_{q^\bullet_\bullet\to(-1)^{r_\bullet-1}q^\bullet} = J_{\Fl,\bbS}(\epsilon)\cup\omega.   
    \]
    
    Consider the map
    \begin{equation}\label{eqn:splitting}
        H^*_\bbS(\FlG)[\![q^\bullet,\epsilon]\!] \oplus \ker(\omega\star^{\FlT}-) \to H^*_\bbS(\FlT)^\W[\![q^\bullet,\epsilon]\!], \quad (\epsilon,\eta)\mapsto \widetilde{\varphi}(\epsilon,\eta) := \varphi(\epsilon)+\eta.
    \end{equation}
    Modulo $q^\bullet$, it coincides with the splitting $H^*_\bbS(\FlG)\oplus\ker(\omega\cup-)\cong H^*_\bbS(\FlT)^\W$ induced from the lift $\varphi_0$, hence, by Nakayama's Lemma, it is an isomorphism.
    Via the splitting \eqref{eqn:splitting}, there is a unique map (over $\C[\![q^\bullet,\sigma]\!]$)
    \[
    H^*_\bbS(\FlT)^\W \to H^*_\bbS(\FlG), \qquad \sigma = (\sigma',\eta) \mapsto \epsilon(\sigma):=\widetilde{\varphi}^{-1}(\sigma',0).
    \]
    This gives the desired mirror map.
\end{proof}

\begin{corollary}\label{cor:CFKS}
    For any $f\in H^*_\bbS(\FlT)^{-\W}_\loc[z][\![q^\bullet,\sigma]\!]$ with $f|_{q=\sigma=0}=\omega$, the function $(z\cdot M_{\FlT,\bbS}(\sigma)f) \cup \omega^{-1}$ lies in $\cL_{\FlG,\bbS}$.
\end{corollary}

\begin{proof}
    Since, for any $j\in J$,
    \[
    z\partial_{\sigma^j}M_{\FlT,\bbS}(\sigma)\omega = M_{\FlT,\bbS}(\sigma) (\omega \cup P_j(H^\bullet_\bullet) + O(q^\bullet)),
    \]
    the function $M_{\FlT,\bbS}(\sigma)f$ can be expressed as a linear combination of derivatives of $M_{\FlT,\bbS}(\sigma)\omega$.
    Hence, $(M_{\FlT,\bbS}(\sigma)f)\cup\omega^{-1}$ can be expressed as a linear combination of derivatives of $(M_{\FlT,\bbS}(\sigma)\omega)\cup\omega^{-1}=J_{\Fl,\bbS}(\epsilon(\sigma))$.
    Since $\cL_{\Fl,\bbS}$ is overruled, it implies that $(z\cdot M_{\FlT,\bbS}(\sigma)f) \cup \omega^{-1}$ lies in $\cL_{\FlG,\bbS}$.
\end{proof}

\begin{lemma} \label{lemma:ab/nonab for flag}
Theorem~\ref{thm:abelian/nonabelian_correspondence} holds in the case when $B$ is a point. 
\end{lemma}

\begin{proof} 
Let $I(q^\bullet_\bullet,t_\bullet^\bullet,x,z)$ be an $H^*_\bbS(\FlT)$-valued formal function with the following properties:
\begin{itemize}
    \item 
    $I(q^\bullet_\bullet,t_\bullet^\bullet,x,z)$ satisfies the Divisor Equation~\eqref{eqn:Divisor_Equation};
    \item 
    $z\cdot I(q^\bullet_\bullet,t_\bullet^\bullet,x,z)$ lies in $\cL_{\FlT,\bbS}$;
    \item 
    the specialization $I(q^\bullet_\bullet,t^\bullet_\bullet,x,z)|_{t_\bullet^\bullet\to t^\bullet}$ is $\W$-invariant.
\end{itemize}
The first and third properties imply that 
\begin{equation}\label{eqn:z_partial_omega_I}
    (I(q^\bullet_\bullet,t^\bullet,x,z))_\gt = (z\partial_\omega I) ((-1)^{r_\bullet-1}q^\bullet,t^\bullet,x,z) \cup \omega^{-1}
\end{equation}
where $z\partial_\omega:=\prod_{\alpha\in\Phi_+} z\partial_{c_1(L_\alpha)}$ and $\partial_{c_1(L_\alpha)}$ denotes the $\C$-linear combination of $\partial_{t_\bullet^\bullet}$ corresponding to $c_1(L_\alpha)$, i.e., the differential operator in $t_\bullet^\bullet$ satisfying 
\[
\partial_{c_1(L_\alpha)} \left( \sum_{m=1}^l \sum_{i=1}^{r_m} t^{(m)}_i H^{(m)}_i \right) = c_1(L_\alpha).
\]
Meanwhile, the second property of $I$ implies that there exist 
\begin{align*}
    \sigma(q^\bullet_\bullet,t^\bullet_\bullet,x) &\in H^*_\bbS(\FlT)_\loc[\![q^\bullet_\bullet,t^\bullet_\bullet,x]\!], \\
    f(q^\bullet_\bullet,t^\bullet_\bullet,x,z) &\in H^*_\bbS(\FlT)_\loc[z][\![q^\bullet_\bullet,t^\bullet_\bullet,x]\!]
\end{align*}
with $\sigma(0,0,0)=0$, $f(0,0,0,z)=1$ such that
\[
I(q^\bullet_\bullet,t_\bullet^\bullet,x,z) = M_{\FlT,\bbS}(\sigma(q^\bullet_\bullet,t^\bullet_\bullet,x)) f(q^\bullet_\bullet,t^\bullet_\bullet,x,z).
\]
The third property of $I$ shows that $\sigma((-1)^{r_\bullet-1}q^\bullet,t^\bullet,x)$ is $\W$-invariant and $f((-1)^{r_\bullet-1}q^\bullet,t^\bullet,x,z)$ is $\W$-invariant.
Since $z\partial_{t^{(m)}_i}\circ M_{\FlT,\bbS}(\sigma) = M_{\FlT,\bbS}(\sigma) \circ (\sigma^*\nabla)_{z\partial_{t^{(m)}_i}}$, the right-hand side of \eqref{eqn:z_partial_omega_I} equals
\[
(M_{\FlT,\bbS}(\sigma((-1)^{r_\bullet-1}q^\bullet,t^\bullet,x)) g((-1)^{r_\bullet-1}q^\bullet,t^\bullet,x,z))\cup\omega^{-1}
\]
for some $g((-1)^{r_\bullet-1}q^\bullet,t^\bullet,x,z)\in H^*_\bbS(\FlT)^{-W}_\loc[z][\![q^\bullet,t^\bullet,x]\!]$ satisfying
$g(0,0,0,z)=\omega$.
Finally, Corollary~\ref{cor:CFKS} yields that 
\[
(I(q^\bullet_\bullet,t^\bullet,x,z))_\gt=(M_{\FlT,\bbS}(\sigma((-1)^{r_\bullet-1}q^\bullet,t^\bullet,x)) g((-1)^{r_\bullet-1}q^\bullet,t^\bullet,x,z))\cup\omega^{-1}
\]
represents a point in $\cL_{\FlG,\bbS}$.
\end{proof}

\begin{remark} \label{rem:CFKS implies CLS}
    The arguments which deduce Lemma \ref{lemma:ab/nonab for flag} from Proposition~\ref{prop:correspondence_J-function} and its Corollary~\ref{cor:CFKS}, and deduce those, in turn, from \cite[Theorem 4.1.1]{CFKS} work for general GIT quotients. In other words, we spelled out how the Coates-Lutz-Shafi formulation \cite[Conjecture 1.8]{CLS} of the abelian/nonabelian correspondence follows from the original \cite[Conjecture 3.7.1]{CFKS}.
    Note that restricting to points in the Givental cone satisfying the Divisor Equation is used precisely to deduce equation \eqref{eqn:z_partial_omega_I}.
\end{remark}

We now remove the assumption on $B$ and prove Theorem~\ref{thm:abelian/nonabelian_correspondence}.
Choose a basis $\{\phi_i\}_{i\in I}$ of $H^*(B)$ such that $\phi_0=1$.
Let $\{\newsigma^{i,j}\}_{(i,j)\in I\times J}$ denote the coordinates on $H^*_\bbS(B\times\FlT)^\W$ associated with the basis $\{\phi_i\otimes P_j(H^\bullet_\bullet)\}_{(i,j)\in I\times J}$.
By Remark \ref{rem:CFKS implies CLS}, it suffices to show the following proposition, the extension of Proposition~\ref{prop:correspondence_J-function} to the case of an arbitrary product $B\times\Fl$.

\begin{proposition}\label{prop:correspondence_J-function_product}
    There exists $\varepsilon(\newsigma)\in H^*_\bbS(B\times\Fl)[\![q^\bullet,Q_B,\newsigma]\!]$ such that 
    \begin{equation}\label{eqn:correspondence_J-function_product}
        [M_{B\times\FlT,\bbS}(\newsigma)\omega]_{q^{(m)}_i\to(-1)^{r_m-1} q^{(m)}} = (M_{B\times\FlG,\bbS}(\varepsilon(\newsigma))1) \cup \omega.
    \end{equation} 
\end{proposition}

\begin{proof}
    Let $J'\subset J$ denote the subset which gives the basis $\{P_j(H_\bullet^\bullet)\}_{j\in J'}$ of $H^2_\bbS(\FlT)^\W$.
    Introduce the differential operator
    \[
    D(s,z\partial_\tau,z\partial_{t_\bullet^\bullet})=\sum_{(i,j)\in K} s^{i,j} z\partial_{\tau^i} \cdot P_j(z\partial_{t_\bullet^\bullet})
    \]
    where $K:= I\times J\setminus(I\times\{0\}\cup\{0\}\times J')$, and $\{s^{i,j}\}_{(i,j)\in K}$ are formal parameters.
    Starting with the explicit small $I$-function \eqref{eqn:brown_function}, we define the $H^*_\bbS(B\times\FlT)$-valued function
    \begin{equation}\label{eqn:I^S}
    \I(q^\bullet_\bullet,t^\bullet_\bullet,\tau,s,z) := e^{D(s,z\partial_\tau,z\partial_{t_\bullet^\bullet})/z}J_B(\tau,z)I_{\FlT}^\bbS(t_\bullet^\bullet,z).
    \end{equation}
    The function $z\cdot \I$ lies in $\cL_{B\times\FlT,\bbS}$ as the differential operator $\exp(D(s,z\partial_\tau,z\partial_{t_\bullet^\bullet})/z)$ preserves the Givental cone (Lemma \ref{lem:differential_operator_preserves_cone}).
    It is straightforward to see that $\I$ satisfies the Divisor Equation~\eqref{eqn:Divisor_Equation}. By Birkhoff factorization (see \cite{CG:quantum}),  there exist 
    \begin{align*}
        \widetilde{\newsigma}(q^\bullet_\bullet,Q_B,t^\bullet_\bullet,\tau,s) &\in H^*_\bbS(B\times\FlT)[\![q^\bullet_\bullet,Q_B,t^\bullet_\bullet,\tau,s]\!], \\
        f(q_\bullet^\bullet,Q_B,t^\bullet_\bullet,\tau,s,z) &\in H^*_\bbS(B\times\FlT)[z][\![q_\bullet^\bullet,Q_B,t^\bullet_\bullet,\tau,s]\!]
    \end{align*} 
    such that $\I(q^\bullet_\bullet,t^\bullet_\bullet,\tau,s,z) = M_{B\times\FlT,\bbS}(\widetilde{\newsigma})f$. Furthermore, one easily checks that $\widetilde{\newsigma}$ and $f$ satisfy
    \begin{align*}
        \widetilde{\newsigma}(0,0,t^\bullet_\bullet,\tau,s)\mid_{t^\bullet_\bullet\to t^\bullet} &= \sum_{j\in J'} t^j\cdot1\otimes P_j(H^\bullet_\bullet) + \sum_{i\in I}\tau^i\cdot\phi_i\otimes1 + \sum_{(i,j)\in K}s^{i,j}\cdot\phi_i\otimes P_j(H_\bullet^\bullet), \\
        f|_{q^\bullet_\bullet=Q_B=0} &= 1.
    \end{align*}
 We now make the following 
\vspace{10pt}
    
    {\bf Claim:} $z\I_\gt^\bbS:=z(\I(t^\bullet,\tau,s,z))_\gt$ lies in $\cL_{B\times\Fl,\bbS}$.

\vspace{10pt}
\noindent Granting the claim for a moment, the theorem follows. 
    Indeed, $\I_\gt^\bbS\cup\omega$ equals\footnote{Here and in the rest of the proof we use liberally Notation \ref{cup omega}}
    \begin{align*}
        &[z\partial_\omega M_{B\times\FlT,\bbS}(\widetilde{\newsigma})f]_{q^\bullet_\bullet\to (-1)^{r_\bullet-1}q^\bullet,t^\bullet_\bullet\to t^\bullet} \\
        =\ &\left[M_{B\times\FlT,\bbS}(\widetilde{\newsigma}) \left(\prod_{\alpha\in\Phi_+}\widetilde{\newsigma}^*\nabla_{z\partial_{c_1(L_\alpha)}}\right)f\right]_{q^\bullet_\bullet\to (-1)^{r_\bullet-1}q^\bullet,t^\bullet_\bullet\to t^\bullet} \\
        =\ &M'_{B\times\FlT,\bbS}(\widetilde{\newsigma}((-1)^{r_\bullet-1}q^\bullet,Q_B,t^\bullet,\tau,s)) g((-1)^{r_\bullet-1}q^\bullet,Q_B,t^\bullet,\tau,s,z)
    \end{align*}
    where $M'_{B\times\FlT,\bbS}:=[M_{B\times\FlT,\bbS}]_{q^\bullet_\bullet\to (-1)^{r_\bullet-1}q^\bullet}$, and $g$ is anti-$\W$-invariant with $g|_{q^\bullet=Q_B=0}=\omega$.
    A similar computation yields the following:
    \begin{align*}
        (z\partial_{\tau^i}\I_\gt^\bbS)\cup\omega &= M'_{B\times\FlT,\bbS}(\widetilde{\newsigma}) g_{i,0}, \quad g_{i,0}|_{q^\bullet=Q_B=0} = \phi_i\otimes\omega &(i\in I) \\
        (z\partial_{t^{j}}\I_\gt^\bbS)\cup\omega &= M'_{B\times\FlT,\bbS}(\widetilde{\newsigma}) g_{0,j}, \quad g_{0,j}|_{q^\bullet=Q_B=0} = \omega \cup P_j(H^\bullet_\bullet) &(j\in J') \\
        (z\partial_{s^{i,j}}\I_\gt^\bbS)\cup\omega &= M'_{B\times\FlT,\bbS}(\widetilde{\newsigma}) g_{i,j}, \quad g_{i,j}|_{q^\bullet=Q_B=0} = \phi_i\otimes(\omega \cup P_j(H^\bullet_\bullet)) &((i,j)\in K)
    \end{align*}
    The left-hand side of \eqref{eqn:correspondence_J-function_product}, $M'_{B\times\FlT,\bbS}(\newsigma)\omega$, can be expressed as a $\C[\![q^\bullet,Q_B,\newsigma]\!]$-linear combination of the derivatives of $\I^\bbS_\gt\cup\omega$ with the change of variable $\newsigma\mapsto(t(\newsigma), \tau(\newsigma), s(\newsigma))$ satisfying $\widetilde{\newsigma}(t(\newsigma),\tau(\newsigma),s(\newsigma))=\sum_{(i,j)\in I\times J} \newsigma^{i,j} \phi_i\otimes P_j(H_\bullet^\bullet)$.
    This shows that $z\I_\gt^\bbS\in\cL_{B\times\Fl,\bbS}$ implies $(z\cdot M'_{B\times\FlT,\bbS}(\newsigma)\omega)\cup\omega^{-1}\in\cL_{B\times\Fl,\bbS}$.
    By taking Birkhoff factorization, one obtains $\epsilon(\newsigma)\in H^*_\bbS(B\times\Fl)[\![q^\bullet,Q_B,\newsigma]\!]$ and $h\in H^*_\bbS(B\times\Fl)[z][\![q^\bullet,Q_B,\newsigma]\!]$ such that
    \[
    (M'_{B\times\FlT,\bbS}(\newsigma)\omega)\cup\omega^{-1} = M_{B\times\Fl,\bbS}(\epsilon(\newsigma))h.
    \]
    The left-hand side of the last equation is of the form $1+O(1/z)$. From this and $M_{B\times\Fl,\bbS}|_{z=\infty}=\operatorname{id}$, it follows that $h=1$. 
    This proves the identity \eqref{eqn:correspondence_J-function_product}.
    
    It remains to prove the Claim that $z\I_\gt^\bbS\in\cL_{B\times\Fl,\bbS}$. To do this, we check that $\I_\gt^\bbS$ satisfies the properties $\mathbf{(C1)}$, $\mathbf{(C2)}$, and $\mathbf{(C3)}$ introduced in Theorem~\ref{thm:characterization}.
    We set 
    \begin{equation*}
        s^{i,j} = 
        \begin{cases}
        \tau^i  & \text{if $j=0$}, \\
        t^{(j)} & \text{if $i=0$ and $j\in J'$},   \\
        s^{i,j} & \text{if $(i,j)\in K$}.
        \end{cases}
    \end{equation*}

    By a direct computation, the Laurent expansion of $\iota_\alpha^*(\I^\bbS_\gt)$ at $z=0$ is written as $\Delta(z\partial_\tau,s,z)\cdot J_B(\tau_\alpha,z)$ where 
    \[
    \tau_\alpha := \sum_{(i,j)\in I\times J} s^{i,j}\cdot\phi_i \cdot P_j(\iota_\alpha^*H_\bullet^\bullet)
    \]
    and $\Delta(z\partial_\tau,s,z) \in H^*_\bbS(B)_\loc[\![z,q^\bullet,z\partial_\tau,s]\!]$ is the following differential operator:
    \begin{multline*}
    \sum_{ \substack{k^{(m)}\geq0; \\1\leq m\leq l} } \left[ \prod_{m=1}^l (q^{(m)}e^{t^{(m)}})^{k^{(m)}} \right] \cdot \sum_{ \substack{k_i^{(m)}\geq0; 1\leq m\leq l, \\ 1\leq i\leq r_m,\sum_ik_i^{(m)}=k^{(m)}} } I^\bbS_{k_\bullet^\bullet}(\iota_\alpha^*H_\bullet^\bullet,z) \cdot W_{k_\bullet^\bullet}^\bbS(\iota_\alpha^*H_\bullet^\bullet,z) \\
    \cdot \exp\left( \frac{D(s,z\partial_\tau,\iota_\alpha^*H_\bullet^\bullet+k_\bullet^\bullet z)-D(s,z\partial_\tau,\iota_\alpha^*H_\bullet^\bullet)}{z} \right).
    \end{multline*}
    Since $\Delta|_{q^\bullet=s=0}=1$, it follows from Lemma~\ref{lem:differential_operator_preserves_cone} that $z\cdot\Delta(z\partial_\tau,s,z)\cdot J_B(\tau_\alpha,z)$ lies in $\cL_B$, i.e., $\I^\bbS_\gt$ satisfies $\mathbf{(C3)}$.

    Next, it is immediate to see that $\I^\bbS_\gt$ satisfies $\mathbf{(C1)}$ and $\mathbf{(C2)}$ if and only if the coefficients of its formal power series expansion in the $s$-variables satisfy $\mathbf{(C1)}$ and $\mathbf{(C2)}$.
    We only check that the coefficient of $s^{i,j}s^{i',j'}$ satisfies $\mathbf{(C1)}$ and $\mathbf{(C2)}$; the other cases follow similarly.
    From the construction, the coefficient of $s^{i,j}s^{i',j'}$ in $\I^\bbS_\gt$ is
    \begin{align*}
        & \left( \left[ \left( \partial_{\tau^i}\partial_{\tau^{i'}}J_B(\tau,z) \right) \cdot P_j(z\partial_{t_\bullet^\bullet}) P_{j'}(z\partial_{t_\bullet^\bullet}) I_{\FlT}^\bbS \right]_{t_i^{(m)}\to t^{(m)}} \right)_\gt  \\
        =&\ \left( \partial_{\tau^i}\partial_{\tau^{i'}}J_B(\tau,z) \right) \cdot \left( \left[ P_j(z\partial_{t_\bullet^\bullet}) P_{j'}(z\partial_{t_\bullet^\bullet}) I_{\FlT}^\bbS \right]_{t_i^{(m)}\to t^{(m)}} \right)_\gt.
    \end{align*}
    Here, $(\cdot)_\gt$ in the second line denotes the $\gt$-modification for $\FlG$.
    
    Since $P_j(z\partial_{t_\bullet^\bullet}) P_{j'}(z\partial_{t_\bullet^\bullet}) I_{\FlT}^\bbS$ lies in a tangent space of $\cL_{\FlT,\bbS}$, Lemma~\ref{lemma:ab/nonab for flag} implies that the second factor
    \[
    \left( \left[ P_j(z\partial_{t_\bullet^\bullet}) P_{j'}(z\partial_{t_\bullet^\bullet}) I_{\FlT}^\bbS \right]_{t_i^{(m)}\to t^{(m)}} \right)_\gt
    \]
    lies in a tangent space of $\cL_{\FlG,\bbS}$ as well. Therefore, by
    Theorem~\ref{thm:characterization}, it satisfies the conditions $\mathbf{(C1)}$ and $\mathbf{(C2)}$ for $\pt\times\FlG$.
    Note that the first factor $\partial_{\tau^i}\partial_{\tau^{i'}}J_B(\tau,z)$ does not contain the $\bbS$-equivariant parameters $\nu_j$. Together with Remark~\ref{rmk:characterization}~(1), this allows us to conclude that the coefficient of $s^{i,j}s^{i',j'}$ satisfy $\mathbf{(C1)}$ and $\mathbf{(C2)}$ for $B\times\FlG$.
\end{proof}

\section{Proof of Lemma \ref{lem:twisted_function}}\label{sec:proof_of_Lemma}
After some preparations, we will prove Lemma \ref{lem:twisted_function}, which finishes the proof of the main theorem (Theorem \ref{thm:mirror_Fl}).
Throughout this section, we fix a natural number $N$ and a subset $r_\bullet$ of $\{1,2,\dots,N-1\}$ with $r_1<\cdots<r_l$, and follow Notation~\ref{notation:flag_bundle} and Notation \ref{notation:S-action}~(1).

\subsection{Reduction to the abelian quotient}\label{reduction1}
We reduce Lemma \ref{lem:twisted_function} to an analogous problem for the abelian quotient.

We introduce the $H^*(B\times\FlT)$-valued function 
\begin{multline}\label{eqn:F_T(mu)}
    \Fab(\mu) := \sum_{ \substack{k_i^{(m)}\geq0; \\ 1\leq m\leq l,  1\leq i\leq r_m} } \left[ \prod_{m=1}^l \prod_{i=1}^{r_m} (q^{(m)}_i)^{k^{(m)}_i}e^{\left(k^{(m)}_i+\frac{H^{(m)}_i}{z}\right)t^{(m)}_i} \right] \cdot I^{\mu+H_\bullet^\bullet+k_\bullet^\bullet z}_{\widetilde{V}}(x,z) \\ 
    \cdot \prod_{i=1}^{r_l}\prod_{c=1}^{k_i^{(l)}}\frac{ \prod_{ \substack{\epsilon\colon\mathrm{Chern\ roots} \\ \mathrm{of\ } \mathcal{Q}} } (\mu+H_i^{(l)}+\epsilon+cz) }{(H_i^{(l)}+cz)^N} \cdot \prod_{m=1}^{l-1}\prod_{i=1}^{r_m}\prod_{j=1}^{r_{m+1}}\frac{\prod_{c=-\infty}^{0}(H_i^{(m)}-H_j^{(m+1)}+cz)}{\prod_{c=-\infty}^{k_i^{(m)}-k_j^{(m+1)}}(H_i^{(m)}-H_j^{(m+1)}+cz)}.
\end{multline}
Observe that $\Fab(\mu)$ satisfies the Divisor Equation~\eqref{eqn:Divisor_Equation}, $[ \Fab(\mu) ]_{t^{(m)}_i\to t^{(m)}}$ is $\W$-invariant, and that it is the abelian counterpart of $F(\mu)$ in the sense that
\begin{equation}\label{eqn:F(mu)-F_T(mu)}
    F(\mu) = \left( \left[ \Fab(\mu) \right]_{t^{(m)}_i\to t^{(m)}} \right)_\gt.
\end{equation}

By Corollary \ref{cor:twisted_abelian_nonabelian_correspondence} applied to $E_\G=\mathcal{Q}\boxtimes S^\vee$ and $E_\T=\mathcal{Q}\boxtimes S^\vee_\T$, we conclude that it suffices to show that $z\cdot \Fab(\mu)\in\cL_{B\times\FlT,(\mathcal{Q}\boxtimes S^\vee_\T,e_\mu)}$.

\subsection{Reduction to a specific choice of $I_{\widetilde{V}}^{\lambda_\bullet^\bullet}$} 
\label{reduction2}
We note that $F(\mu)$ and $\Fab(\mu)$ depend on the choice of $I_{\widetilde{V}}^{\lambda_\bullet^\bullet}(x,z)$. We show next that proving $z\cdot \Fab(\mu)\in\cL_{B\times\FlT,(\mathcal{Q}\boxtimes S^\vee_\T,e_\mu)}$ for a specific convenient choice implies it for all other choices.   

Fix a $\C$-basis $\{\phi_a\}_{a\in I}$ of $H^*(B)$ and a $\C$-basis $\{g_0=1,g_1(\lambda_\bullet^\bullet),g_2(\lambda_\bullet^\bullet),\dots\}$ of $H^*_\T(\pt)^\W$.
Then, the set $\{\phi_az^bg_c\}_{a\in I, b\geq0, c\geq0}$ forms the $\C$-basis of $H^*_\T(\widetilde{V})^\W[z]$, and let $\widetilde{\btau}=\{\tau^{a,b,c}\}_{a\in I,b\geq0,c\geq0}$ denote the dual coordinates.
We define $\J_{\widetilde{V}}^{\lambda_\bullet^\bullet}(\widetilde{\btau})$ to be the $H^*_\T(\pt)[\![Q_B,\widetilde{\btau}]\!]$-valued $\W$-invariant point on $z^{-1}\cL_{\widetilde{V}}$ of the form
\[
z\cdot \J_{\widetilde{V}}^{\lambda_\bullet^\bullet}(\widetilde{\btau}) = z + \sum_{a\in I,b\geq0,c\geq0} \tau^{a,b,c}\phi_az^bg_c + O(z^{-1}).
\]
In fact, this point is \emph{miniversal}, i.e., for any $\W$-invariant point $z\cdot I(x)$ on $\cL_{\widetilde{V}}$, there exist $\tau^{a,b,c}(x)\in\C[\![Q_B,x]\!]$ ($a\in I,b\geq0,c\geq0$) such that $\J_{\widetilde{V}}(\widetilde{\btau}(x))=I(x)$. 
In other words, every $\W$-invariant point on $\cL_{\widetilde{V}}$ can be written as a pullback of $\J_{\widetilde{V}}^{\lambda_\bullet^\bullet}(\widetilde{\btau})$.
Hence we may assume that $I_{\widetilde{V}}^{\lambda_\bullet^\bullet}=\J_{\widetilde{V}}^{\lambda_\bullet^\bullet}(\widetilde{\btau})$.

Let $\btau=\{\tau^{a,b,0}\}_{a\in I,b\geq0}$ denote the coordinates on $H^*(\widetilde{V})[z]$.
The following lemma reduces Lemma \ref{lem:twisted_function} to the case where $I_{\widetilde{V}}^{\lambda_\bullet^\bullet}=\J_{\widetilde{V}}^{\lambda_\bullet^\bullet}(\btau)$.

\begin{lemma}\label{lemma_reduction2}
    Let $F_{\operatorname{ab},1}(\mu,\btau,z)$ (resp. $F_{\operatorname{ab},2}(\mu,\widetilde{\btau},z)$) be the function \eqref{eqn:F_T(mu)} with $I_{\widetilde{V}}^{\lambda_\bullet^\bullet}=\J_{\widetilde{V}}^{\lambda_\bullet^\bullet}(\btau)$ (resp. $I_{\widetilde{V}}^{\lambda_\bullet^\bullet}=\J_{\widetilde{V}}^{\lambda_\bullet^\bullet}(\widetilde{\btau})$).
    Let $F(\mu,\widetilde{\btau},z)$ be the function \eqref{eqn:twisted_F} with $I_{\widetilde{V}}^{\lambda_\bullet^\bullet}=\J_{\widetilde{V}}^{\lambda_\bullet^\bullet}(\widetilde{\btau})$.
    Then,
    \begin{align}
        F_{\operatorname{ab},2}(\mu,\widetilde{\btau},z) &= \Delta(\widetilde{\btau},z\partial_\btau,z\partial_{t_\bullet^\bullet},z) \cdot F_{\operatorname{ab},1}(\mu,\btau,z), \label{eqn:first_claim}\\
        F(\mu,\widetilde{\btau},z) &= \left( \left[ F_{\operatorname{ab},2}(\mu,\widetilde{\btau},z) \right]_{t^{(m)}_i\to t^{(m)}} \right)_\gt,\label{eqn:second_claim}
    \end{align}
    where $\Delta(\widetilde{\btau},z\partial_\btau,z\partial_{t_\bullet^\bullet},z)$ is the differential operator
    \[
    \exp\left( z^{-1} \sum_{a\in I} \sum_{b\geq0} z\partial_{\tau^{a,b,0}} \sum_{c>0} \tau^{a,b,c} \cdot g_c(z\partial_{t_\bullet^\bullet}) \right).
    \]
    In particular, if $z\cdot F_{\operatorname{ab},1}(\mu,\btau,z)$ lies in $\cL_{B\times\FlT,(Q\boxtimes S^\vee_\T,e_\mu)}$, then $z\cdot F(\mu,\widetilde{\btau},z)$ lies in $\cL_{B\times\FlG,(Q\boxtimes S^\vee,e_\mu)}$.
\end{lemma}

\begin{proof}
    Equation \eqref{eqn:first_claim} follows from a direct computation; see the proof of \cite[Lemma 5.5]{Koto:mirror} for a similar computation.
    The second equation \eqref{eqn:second_claim} follows from the equations \eqref{eqn:first_claim} and \eqref{eqn:F(mu)-F_T(mu)}.
\end{proof}

\subsection{Quantum Riemann-Roch operators}
We introduce differential operators $A(\mu,z\partial_{t^{(l)}_i},z)$ ($1\leq i\leq r_l$) and summarize its properties which will be needed.
    
Let $\mu$ and $y$ be formal variables.
Note that $\mu$ will be identified with the equivariant parameter for $\C^\times$ acting on $\mathcal{Q}\boxtimes S^\vee$.
We define $A(\mu,y,z)$ to be
\[
\exp\left( \frac{\rank(\mathcal{Q})}{z} \cdot \left( \left( \mu+y+\frac{z}{2} \right) \cdot \log\left(1+\frac{y}{\mu}\right) - y \right) + \frac{z\partial_\mathcal{Q}}{z} \cdot \log \left( 1+\frac{y}{\mu} \right) \right)
\]
where $\partial_{\mathcal{Q}}$ denotes the vector field on the $(\tau^i\colon i\in I)$-space such that $\partial_{\mathcal{Q}}\sum_{i\in I}\tau^i\phi_i=c_1(\mathcal{Q})$, and $\log(1+y/\mu)$ is considered as an element of $\C[y][\![\mu^{-1}]\!]$ by expanding it at $\mu=\infty$.

\begin{lemma}[{\cite[Lemma 5.6]{Koto:mirror}}]\label{lem:QRR_operators}
    Let $1\leq i\leq r_l$.
    \begin{itemize}
        \item[$(1)$]
        The operator $A(\mu,z\partial_{t^{(l)}_i},z)$ is a well-defined element of $\C[z^{-1},z\partial_\tau,z\partial_{t^{(l)}_i}][\![\mu^{-1}]\!]$ such that $A|_{\mu^{-1}=0}=1$.
        In particular, $A(\mu,z\partial_{t^{(l)}_i},z)$ preserves (twisted) Givental cones.
        \item[$(2)$]
        For any $k\geq0$, we have
        \[
        A(\mu,H^{(l)}_i+kz,z)\cdot\frac{\Delta_{\mathcal{Q}}^{\mu+H_i^{(l)}+kz}}{\Delta_{\mathcal{Q}}^{\mu+H_i^{(l)}}} = \prod_{c=1}^{k}\prod_{ \substack{\epsilon\colon\mathrm{Chern\ roots} \\ \mathrm{of\ } \mathcal{Q}} } (\mu+H_i^{(l)}+\epsilon+cz).
        \]
    \end{itemize}
\end{lemma}

\subsection{Completion of the proof}\label{subsec:finish}
We finish the proof of Lemma \ref{lem:twisted_function} using the reductions discussed in subsections \ref{reduction1} and \ref{reduction2}.
The argument is similar to the proofs of \cite[Theorem 3.3]{IK:quantum} and \cite[Theorem 5.1]{Koto:mirror}.

First, the exact sequence \eqref{eqn:exact_sqn} shows that in the $\T$-equivariant $K$-group of vector bundles on $B$, the class of $\widetilde{V} \oplus \operatorname{Hom}(\mathcal{O}_B^{\oplus r_l},\mathcal{Q})=\widetilde{V} \oplus \mathcal{Q}^{\oplus r_l}$ equals the class of a trivial bundle of the same rank (with nontrivial $\T$-action). Then \cite[Lemma 2.9]{Koto:mirror} implies that
\begin{equation}\label{eqn:function}
    z \cdot \left( \prod_{i=1}^{r_l} \Delta_{\mathcal{Q}}^{\lambda_i^{(l)}} \right) \J_{\widetilde{V}}^{\lambda_\bullet^\bullet}(\btau,z)
\end{equation}
is an $\C[\![Q_B,\btau,(\lambda^{(l)}_1)^{-1},\dots,(\lambda^{(l)}_{r_l})^{-1}]\!]^\W$-valued point on $\cL_B$.
Applying Lemma \ref{lem:moving_points} to the function~\eqref{eqn:function} (with the set of variables $x$ taken as the set of elementary symmetric polynomials in $\{(\lambda^{(l)}_i)^{-1}\}_{i=1}^{r_l}$), we obtain a differential operator $D(\lambda_\bullet^\bullet)\in\sum_{i\in I}\C[z][\![Q_B,\btau,(\lambda^{(l)}_1)^{-1},\dots,(\lambda^{(l)}_{r_l})^{-1}]\!]^\W z\partial_{\tau^i}$ such that
\[
\left( \prod_{i=1}^{r_l} \Delta_{\mathcal{Q}}^{\lambda_i^{(l)}} \right) \J_{\widetilde{V}}^{\lambda_\bullet^\bullet}(\btau,z) = e^{D(\lambda^\bullet_\bullet)/z}J_B(\tau,z),
\]
where $\tau=\{\tau^{i,0,0}\}_{i\in I}$ denotes the coordinates on $H^*(B,\C)$ and $J_B(\tau,z)$ denotes the $J$-function \eqref{eqn:J-function} of $B$.

We consider the $H^*(B\times\FlT)$-valued function 
\[\Fab'(\mu):=
\left(\Delta_{\mathcal{Q}\boxtimes S_\T^\vee}^\mu\right)^{-1} \cdot 
\left( \prod_{i=1}^{r_l}  {A(\mu,z\partial_{t^{(l)}_i},z)} \right) \cdot  e^{D(\mu+z\partial_{t_\bullet^\bullet})/z} I_{B\times\FlT} 
\]
where $I_{B\times\FlT}$ denotes Brown's (non-equivariant) $I$-function recalled in Section \ref{subsec:Brown}.
Using Lemma \ref{lem:QRR_operators} (2), one checks by a direct calculation that $\Fab'(\mu)$ coincides with the function $F_{\operatorname{ab},1}(\mu)$ from Lemma \ref{lemma_reduction2}. 
Theorem \ref{thm:QRR} together with Lemma \ref{lem:QRR_operators} (1) implies that $z\cdot F_{\operatorname{ab},1}(\mu)\in\cL_{B\times\FlT,(\mathcal{Q}\boxtimes S^\vee_\T,e_\mu)}$.
As already noted in $\S 6.1$ and $\S 6.2$, this proves Lemma \ref{lem:twisted_function}.

\bibliographystyle{amsplain}
\bibliography{reference}

\end{document}